\documentclass[a4paper, 12pt]{article}
\usepackage[dvipdfmx]{graphicx}
\usepackage{amsmath,amssymb}
\usepackage{amsmath}
\usepackage{amsfonts}
\usepackage{amssymb}
\usepackage{amsthm}

\usepackage{authblk}
\usepackage{xcolor} 
\usepackage{mathtools}

\usepackage{subfig}

\usepackage{tikz}
\usetikzlibrary{arrows,chains,matrix,positioning,scopes}
\makeatletter
\tikzset{
    >=stealth',
    pil/.style={
           ->,
           shorten <=2pt,
           shorten >=2pt,},         
}
\makeatother

\makeatletter
\newcommand*\rel@kern[1]{\kern#1\dimexpr\macc@kerna}
\newcommand*\overwidebar[1]{%
  \begingroup
  \def\mathaccent##1##2{%
    \rel@kern{0.8}%
    \overline{\rel@kern{-0.8}\macc@nucleus\rel@kern{0.2}}%
    \rel@kern{-0.2}%
  }%
  \macc@depth\@ne
  \let\math@bgroup\@empty \let\math@egroup\macc@set@skewchar
  \mathsurround\z@ \frozen@everymath{\mathgroup\macc@group\relax}%
  \macc@set@skewchar\relax
  \let\mathaccentV\macc@nested@a
  \macc@nested@a\relax111{#1}%
  \endgroup
}
\newcommand*\underwidebar[1]{
  \begingroup
  \def\mathaccent##1##2{%
    \rel@kern{0.2}%
    \underline{\rel@kern{-0}\macc@nucleus\rel@kern{-0.8}}%
    \rel@kern{0.6}%
  }%
  \macc@depth\@ne
  \let\math@bgroup\@empty \let\math@egroup\macc@set@skewchar
  \mathsurround\z@ \frozen@everymath{\mathgroup\macc@group\relax}%
  \macc@set@skewchar\relax
  \let\mathaccentV\macc@nested@a
  \macc@nested@a\relax111{#1}%
  \endgroup
} 
\makeatother

\DeclarePairedDelimiter\abs{\lvert}{\rvert}%
\DeclarePairedDelimiter\norm{\lVert}{\rVert}%
\DeclarePairedDelimiter\pair{\langle}{\rangle}%
\makeatletter
\let\oldabs\abs
\def\abs{\@ifstar{\oldabs}{\oldabs*}}
\let\oldnorm\norm
\def\norm{\@ifstar{\oldnorm}{\oldnorm*}}
\let\oldpair\pair
\def\pair{\@ifstar{\oldpair}{\oldpair*}}
\makeatother

\numberwithin{equation}{section}

\theoremstyle{plain}
\newtheorem{theorem}{Theorem} [section] 
\newtheorem{lemma}[theorem]{Lemma} 
\newtheorem{proposition}[theorem]{Proposition}

\theoremstyle{definition}

\theoremstyle{remark}
\newtheorem{remark}{Remark}[section]

\newcommand{\eps}{\varepsilon}

\usepackage[pagewise]{lineno} 

\title{The potential roles of transacylation in intracellular lipolysis and 
related QSSA approximations}
\author[1,2]{J\'{a}n Elia\v{s}\thanks{Email: janelias@ymail.com}}
\author[2]{Klemens Fellner\thanks{Email: klemens.fellner@uni-graz.at}}
\author[3]{Peter Hofer\thanks{Email: peter.hofer@uni-graz.at}}
\author[3,4,5]{Monika Oberer\thanks{Email: m.oberer@uni-graz.at}}
\author[3,5]{Renate Schreiber\thanks{Email: renate.schreiber@uni-graz.at}}
\author[3,4,5]{Rudolf Zechner\thanks{Email: rudolf.zechner@uni-graz.at}}
\affil[1]{Boehringer Ingelheim RCV GmbH \& Co KG, Dr. Boehringer-Gasse 5-11, 1121 Vienna, Austria}
\affil[2]{Institute of Mathematics and Scientific Computing, University of Graz, Heinrichstrasse 36, 8010 Graz, Austria}
\affil[3]{Institute of Molecular Biosciences, University of Graz, Humboldtstraße 50, 8010 Graz, Austria}
\affil[4]{BioTechMed Graz, 8010 Graz, Austria}
\affil[5]{BioHealth Field of Excellence, University of Graz, 8010 Graz, Austria}
\date{\small \today}                  

\begin{document}

\maketitle

\begin{abstract} 
Fatty acids (FAs) are crucial energy metabolites, signalling molecules, and membrane building blocks for a wide range of organisms. Adipose triglyceride lipase (ATGL) is the first and presumingly most crucial regulator of FA release from triacylglycerols (TGs) stored within cytosolic lipid droplets. However, besides the function of releasing FAs by hydrolysing TGs into diacylglycerols (DGs), ATGL also promotes the transacylation reaction of two DG molecules into one TG and one monoacylglycerol molecule. To date, it is unknown whether DG transacylation is a coincidental byproduct of ATGL-mediated lipolysis or whether it is physiologically relevant. Experimental evidence is scarce since both, hydrolysis and transacylation, rely on the same active site of ATGL and always occur in parallel in an ensemble of molecules. 

This paper illustrates the potential roles of transacylation.
It shows that, depending on the kinetic parameters but also on the 
state of the hydrolytic machinery, transacylation can increase or decrease downstream products up to 50\% and more.
We provide an extensive asymptotic analysis including 
quasi-steady-state approximations (QSSA) with higher order correction terms and provide numerical simulation. We also argue that when assessing the validity of QSSAs one should include parameter sensitivity derivatives. Our results suggest that the transacylation function of ATGL is of biological relevance by providing feedback options and altogether stability to the lipolytic machinery in adipocytes.

\medskip


\medskip


\end{abstract}

\section{Introduction}

Fatty acids (FAs) are crucial for ATP production, synthesis of biological membranes, thermogenesis, and signal transduction \cite{Glatz-2014, Zechner-2012}. Animals and humans store FAs in the form of water-insoluble triacylglycerols (TGs, one species thereof is triolein, TO) within cytosolic lipid droplets (LDs) of specialised fat cells called adipocytes, but also in other cell types. From a biochemical perspective, a TG molecule is composed of three FAs esterified to a glycerol backbone. LD-associated TGs are degraded step-wise by an enzymatic cascade of lipid hydrolases commonly designated as "lipases" in a process called intracellular lipolysis.

Distinct lipases are responsible for each lipolytic step \cite{Grabner-2021, Young-2013}. The initial step in the degradation of TGs is catalysed by adipose triglyceride lipase (ATGL) releasing one FA and one diacylglycerol (DG, one species thereof is diolein, DO) molecule \cite{Schreiber-2019}. In the second lipolytic step, hormone-sensitive lipase (HSL) hydrolyses  DGs into FAs and monoacylglycerols (MGs; one species thereof is monoolein, MO)  \cite{Recazens-2021}. Finally, monoglyceride lipase (MGL) hydrolyses MGs into FAs and glycerol, both of which can be secreted from the adipocytes into the bloodstream to supply remote organs. 

To ensure an adequate response to changes in the metabolic state of an organism, lipolysis is fine-tuned by a large number of molecular factors regulating the three lipases \cite{Hofer-2020}. One of these factors is comparative gene identification-58 (CGI-58) which accelerates lipolysis by increasing ATGL activity \cite{Lass-2006}. Despite literally thousands of studies on the physiological role of lipolysis and its complex regulation, a quantitative description of the intricate reaction processes of TG hydrolysis is, however, still missing.

\medskip

Given the outstanding metabolic importance of lipolysis, it is remarkable that ATGL (and possibly HSL \cite{Zhang-2019}) not only catalyse the hydrolysis reaction but in parallel also a transacylation reaction, which transfers a FA from one DG onto a second DG to generate TG and MG \cite{Jenkins-2004, Zhang-2019, Kulminskaya-2021}. 
In fact, measurement of ATGL activity is far from being trivial due its hydrolase and transacylase activity occurring simultaneously. Chemical inhibition of ATGL \cite{Kulminskaya-2021} as well as mutations in its active site \cite{Patel-2022} abolish both activities, implying that it is experimentally not possible to separate one activity from the other. For example, as soon as a TG molecule formed by DG transacylation is hydrolysed to DG, then the effective outcome (one FA and MG were formed) is identical to the outcome of a DG hydrolysis event, thereby masking DG transacylase activity. 

The importance and potential metabolic function of transacylation 
is unknown. 
By its function, the transacylation reaction provides feedback towards the TG pool
as well as an alternative in creating MG alongside DG hydrolysis.  
One may believe that transacylation plays a minor to no role in TG metabolism.
However, \textit{in vitro} experiments from our laboratory (Figure~\ref{fig:DO-Assay}) and others \cite{Jenkins-2004, Kulminskaya-2021} showed remarkable ATGL-mediated DG transacylation. 
 
\begin{figure}[!htb]  
\centering
	\hspace{-0.5cm} \includegraphics[width=0.55\textwidth]{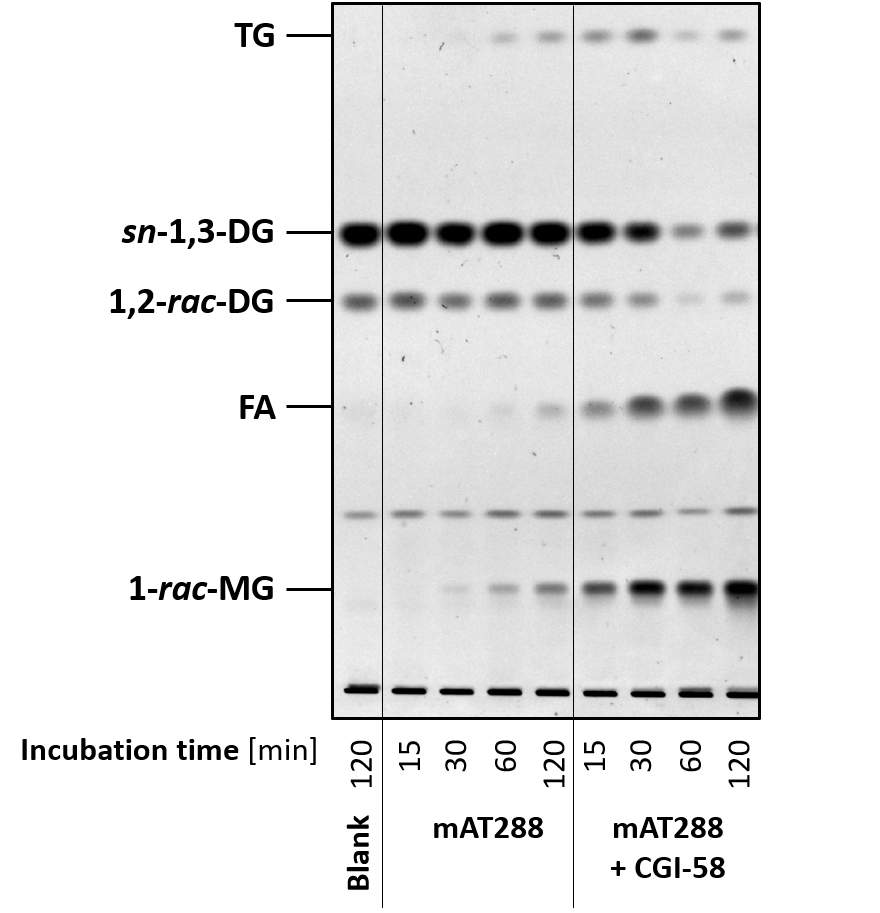} \hspace{-1.5cm}
 	\caption[]{\textbf{Catalytic activities of ATGL towards DGs.} Truncated, purified mouse ATGL covering amino acids 1-288 (mAT288) was incubated with a phospholipid-emulsified, 0.3 mM DO (a particular DG) substrate in the absence and presence of its physiological co-activator CGI-58. Incubation of the DO substrate with an equal volume of buffer served as negative control (blank). At the given time points, reactions were stopped and total lipids were extracted with CHCl\textsubscript{3}:MeOH=2:1. Lipids were resolved by thin layer chromatography using the solvent CHCl\textsubscript{3}:acetone:acetic acid=88:12:1 and visualised by charcoaling at 140$^{\circ}$C.}
  	\label{fig:DO-Assay}
 \end{figure}
 
 The experimental procedure behind obtaining Figure~\ref{fig:DO-Assay} uses a well-defined substrate of artificial LDs  made of DO (a particular DG) and purified ATGL as sole lipase. 
More precisely, a truncated variant of mouse ATGL (mAT288) was used since ATGL is biochemically challenging enzyme 
and removing a functionally non-essential part of ATGL helps it's purification. 
Figure~\ref{fig:DO-Assay} shows that when mAT288
was incubated with an emulsion of DO LDs, significant amounts of TO were formed both in absence or presence of the ATGL endogenous co-activator CGI-58. As  mentioned above, it is impossible to directly distinguish 
DO transacylation and DO hydrolysis in Figure~\ref{fig:DO-Assay}, since TOs produced by transacylation are later hydrolysed by ATGL to DOs and FAs.

\medskip

In this paper we study the potential roles of transacylation in relation to the 
hydrolytic pathway TG $\to$ DG $\to$ MG as sketched in Figure~\ref{fig:reactionnetwork}. 
More precisely, we consider three activities: TG hydrolysis, DG hydrolysis and DG transacylation. We will not include the final step of lipolysis, in which MGL hydrolyses MG to glycerol and FA, since this step is purely downstream and does not interact with 
the question of understanding the role of DG transacylation. 

\begin{figure}
\centering
\begin{tikzpicture}[thick,scale=1.2, every node/.style={transform shape}]

\draw (-0.2,0) node {$TG$};
\draw [->] (0.2,0) -- (1.6,0);
\draw (2,0) node {$DG$};
\draw [->] (2.4,0) -- (3.8,0);
\draw (4.3,0) node {$MG$};
\draw (1,0.2) node {$m_1$}
	 (3.2,0.2) node {$m_2$};
\draw [->] (0.6,0) to [out=+90,in=180] (1.1,+0.6);
\draw [->] (2.8,0) to [out=+90,in=180] (3.3,+0.6);
\draw (1.4,+0.6) node {$FA$}
	 (3.6,+0.6) node {$FA$};
\draw (2,-0.8) node {$2 DG$};
\draw [->] (2,-1.05) to [out=-110,in=0] (0.3,-1.6);
\draw (-0.2,-1.6) node {$TG$};
\draw [->] (2,-1.05) to [out=-70,in=180] (3.6,-1.6);
\draw (4.2,-1.6) node {$MG$};
\draw (2,-1.5) node {$\tau$};
 \end{tikzpicture}
\caption{$m_1$ rate of TG hydrolysis, $m_2$ rate of DG hydrolysis, $\tau$ rate of transacylation.}
\label{fig:reactionnetwork}
\end{figure}
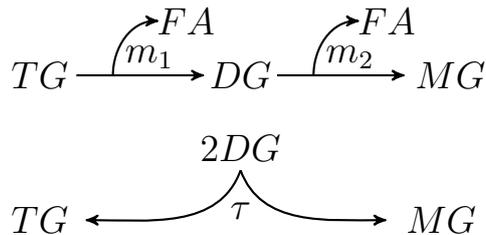


TG hydrolysis is primarily catalyzed by ATGL showing robust TG hydrolase activity especially in presence of CGI-58 \cite{Lass-2006}. Accordingly, ATGL \cite{Haemmerle-2006} as well as CGI-58 knockout mice \cite{Radner-2010} accumulate TGs to supraphysiological levels in various tissues. However, ATGL can also act as  DG transacylase  \cite{Jenkins-2004, Zhang-2019, Kulminskaya-2021} and DG hydrolase \cite{Zimmermann-2004}, indicating that it is also involved in the second lipolytic step. Nevertheless, DG hydrolysis is believed to be primarily catalysed by HSL. Adding to the complexity of lipolysis, HSL also exhibits minor but empirically proven TG hydrolase activity \cite{Fredrikson-1981}.  

For the sake of simplicity and clarity, we will not detail the activities of ATGL and HSL, whose relative contributions are 
also not sufficiently well understood. Instead, we will combine the functions together into single reaction rates for TG hydrolysis, 
DG transacylation and DG hydrolysis, see Figure~\ref{fig:reactionnetwork}.
Specifically, we postulate Michaelis-Menten (MM) kinetics for TG and DG hydrolysis and 
mass action law kinetics (MAL) for DG transacylation.

However, both approaches have drawbacks. Firstly, choosing MM kinetics should not be considered as an ultimate 
modelling choice: Despite the fact that MM kinetic rates for TG and DG hydrolysis can be found in the literature (see e.g. \cite{Viertlmayr-2018,Kim-2008}  and the references therein)
already the mere fact that transacylation occurs in parallel to hydrolysis
implies that hydrolysis can at best be approximatively described 
by MM kinetics. 
In favour of our modelling choice, the data in 
\cite{Viertlmayr-2018} also show that MM kinetics for TG hydrolysis 
is still a decent approximation of the real, more complicated dynamics. 

Secondly, using MAL as reaction rate model for an enzymatic transacylation reaction
is also questionable. However, we compared parameter fittings 
for preliminary \textit{in vitro} experimental data using mAT288 incubated with a phospholipid-emulsified TG substrate and found that using MM or Hill kinetics 
instead of MAL didn't improve the (already good) fitting quality obtain with the MAL model, and despite 
involving two resp. three fitting parameters instead of one. Hence, Occam's razor compels us to use a MAL transacylation rate, at least until the molecular 
mechanism of transacylation becomes better understood. 

From a mathematical viewpoint, the stoichiometry of DG hydrolysis (involving one DG molecule) is different to DG transacylation (involving two DG molecules), which implies different nonlinearities of the respective reaction rates. Hence, the relative contributions
of DG hydrolysis and DG transacylation  
depend on the DG concentration, which in return depends on 
the kinetics producing and processing DGs: We are faced with a complex  nonlinear system behaviour. 
Finally, we also emphasise that we believe that the analysis of this paper can be extended to general
superlinear and monotone increasing models for DG transacylation rates.

\medskip
 
Let us denote the substrate concentrations $s = [TG]$, 
the intermediate concentration $q = [DG]$, the product concentration $p = [MG]$ and the concentration of fatty acids $f = [FA]$. Then, we consider 
the reaction rates
\begin{equation}
m_1(s) = \frac{V_1 s}{K_1+s} \quad \text{and} \quad m_2(q) = \frac{V_2 q}{K_2+q}
\quad \text{and} \quad \tau(q) = \sigma q^2,
\end{equation}
where $V_1$, $K_1$ and $V_2$, $K_2$ are the corresponding maximum velocities and Michaelis constants of the TG and DG hydrolysis 
and $\sigma$ denotes the MAL constant for the transacylation 2DG $\to$ 
TG + MG. Note that it is plausible to assume constant concentrations of ATGL and HSL. Hence, there is no explicit dependence of the parameter  $V_1$, $\sigma$ and $V_2$ on ATGL or HSL.

\medskip

The system of kinetic equations is as follows
\begin{equation}\label{eq:01}
\begin{cases}
\begin{aligned}
     	&\dot{s} = - m_1 + \tau &&= -\frac{V_1 s}{K_1 + s} + \sigma q^2,\\
    	&\dot{q} = m_1 - m_2 - 2 \tau &&= \frac{V_1 s}{K_1 + s} -\frac{V_2 q}{K_2 + q} - 2 \sigma q^2, \\
    	&\dot{p} = m_2 + \tau &&= \frac{V_2 q}{K_2 + q} + \sigma q^2,\\
	&\dot{f} = m_1 + m_2 &&= \frac{V_1 s}{K_1 + s} + \frac{V_2 q}{K_2 + q}, \\
\end{aligned}
\end{cases}
\end{equation}
and we assume that
\[s(0) = s_{0} > 0, \quad q(0) = q_0 \ge 0, \quad p(0) = 0 \quad \text{and} \quad f(0) = 0.\]
We remark that $p$ and $f$ are downstream quantities, which are calculated 
from the solutions of the closed equations for $s$ and $q$. Therefore, a big part of our analysis will focus on the first two equations. 

The substrate TG will eventually be fully hydrolysed into products MG and FA through the processes in the reaction scheme in Figure~\ref{fig:reactionnetwork}. It is fully hydrolysed even in the case when $\tau = 0$. The cascade of MM enzyme reactions in Figure~\ref{fig:reactionnetwork} with $\tau = 0$ has been previously studied, \cite{Storer-1974}. 

In either cases, the total concentrations of fatty acid and glycerol moieties are conserved in time, which leads to two conservations laws
\begin{equation}\label{eq:masscons}
3s + 2q + p + f = 3 s_0 + 2 q_0 \qquad \text{and} \qquad s + q + p = s_0 + q_0 
\end{equation}
for all times. As a consequence of the conservation laws and $s, q \to 0$ as time $t \to \infty$, 
system \eqref{eq:01} converges at $t\to\infty$ to the following equilibrium state:
\begin{equation}\label{equilibrium}
 s,q \xrightarrow{t\to\infty}0, \qquad  p\xrightarrow{t\to\infty} p_{\infty} := s_0 + q_0, \qquad  f \xrightarrow{t\to\infty} f_{\infty} := 2s_0 + q_0.
 \end{equation}
 
\subsection{Outline}
In Section~\ref{sec:scaling}, we non-dimensionalise model system \eqref{eq:01} by introducing a suitable scaling. Moreover, we discuss the general behaviour and illustrate it by plotting a representative phase portrait. 

The results of this article are summarised as follows: 
In Section~\ref{sec:role}, we specify three perspectives of judging the role of transacylation and provide 
a comparative discussion in terms of numerical simulations.  
As expected from a nonlinear system like \eqref{eq:01} the role of transacylation is complicated. It can both speed up and slow down 
significantly (up to 80\% and even more) the production of FAs and MGs. 
The complexity of the behaviour raises the question, which conditions might allow to use a simplified system as approximation of the full model \eqref{eq:01}. We will also emphasise that the dynamics depends not only on the parameters, but \emph{also on the current state} of the hydrolytic machinery. It is possible, for instance, that DG transacylation initially slows down MG production only to altogether speed it up in the long-run of the evolution.

In Section~\ref{sec:QSSA}, we exploit that $q$ chances according to one gain and two loss terms, i.e. $\dot q = m_1 - m_2 -2\tau$.
We discuss in detail the involved time scales and provided Proposition~\ref{prop:QSSA}, which 
states conditions for a QSSA to be reasonable justified. More precisely, we identify three non-dimensional parameters, which -- when assumed sufficiently large -- allow $q$ to be  
approximated by a QSSA.

In Section~\ref{sec:asymptotics}, for each of these three parameters, we derive  the corresponding first order QSSA models.  
We illustrate these approximations with numerical simulations including some higher order correction terms.

In the final Section~\ref{sec:conclusion}, we provide a comprehensive discussion of the findings of the paper.
We shall point out that DG transacylation can serve a failsafe mechanism, which produces MGs and FAs in the case of 
a lacking DG hydrolysis. We will show that this benefit comes at the price of slowing down the 
hydrolytic machinery in cases when DG transacylation and DG hydrolysis compete for substrate. 
We will also point out an example case, where the numerical plot of $q(t)$ appears to the eye as already well approximated by its QSSA. Yet, the sensitivity derivative w.r.t. the DG transacylation rate constant $\kappa$,  which is a qualitative measure of the change due to variations in $\kappa$,  
is predicted wrongly by the QSSA reduced system. 
Hence, as future works, we suggest to include parameter sensitivity derivatives 
into the discussion whether a QSSA  is an admissible model reduction.

\section{Nondimensionalisation and qualitative behaviour}\label{sec:scaling}

We non-dimensionalise the enzymatic model system \eqref{eq:01} by rescaling
\[ t \to \frac{t}{T_{\mathrm{ref}}}, \quad s \to \frac{s}{S_{\mathrm{ref}}}, \quad q \to \frac{q}{Q_{\mathrm{ref}}}, \quad p \to \frac{p}{P_{\mathrm{ref}}} \quad \text{and} \quad f \to \frac{f}{F_{\mathrm{ref}}},
\]
where $T_{\mathrm{ref}}$, $S_{\mathrm{ref}}$, $Q_{\mathrm{ref}}$, $P_{\mathrm{ref}}$ and $F_{\mathrm{ref}}$ are the reference time and concentrations. Note that in a slight abuse of notation, we denote 
dimensional and rescaled non-dimensional variables with the same letters. 
Different reference time and concentrations can be chosen. 

A natural choice for $S_{\mathrm{ref}}$ in a model for TG breakdown is the initial TG concentration $s_0$, i.e. $S_{\mathrm{ref}} = s_0>0$. We choose also
$P_{\mathrm{ref}} = F_{\mathrm{ref}} = s_0$ as TG is ultimately processed into MG and FA. Since TG hydrolysis is the first hydrolytic step, which governs all the other activities, we normalise the time scale accordingly, i.e.  $T_{\mathrm{ref}} = s_0/V_1$. This implies that the $s$ equation 
rescales like
\begin{equation*}
\dot{s} = -\frac{V_1 s}{K_1+s} + \sigma q^2 \quad \xrightarrow{\mathrm{non-dimensionalisation}}
\quad \dot{s} = -\frac{s}{K+s} + \frac{\sigma Q_{\mathrm{ref}}^2}{V_1} q^2, \quad\text{where}\quad K := \frac{K_1}{s_0},
\end{equation*} 
where the right hand side stated the rescaled non-dimensional $s$ equation. 
The above choices of reference values lead to the following rescaled version of the $q$ equation:
\begin{equation}\label{eq:scaleq}
\dot q = \frac{s_0}{Q_{\mathrm{ref}}} \frac{s}{K+s} - \frac{s_0}{Q_{\mathrm{ref}}}\frac{V_2}{V_1} \frac{q}{K_2/Q_{\mathrm{ref}} + q} - 2 \frac{s_0}{Q_{\mathrm{ref}}} \frac{\sigma Q_{\mathrm{ref}}^2}{V_1} q^2.
\end{equation}

A preferable choice for $Q_{\mathrm{ref}}$ is less obvious. 
A first idea might be to also chose $Q_{\mathrm{ref}}=s_0$ like for the other concentrations. However, since DGs are processed by DG hydrolysis and DG transacylation alongside their production from DG hydrolysis, the actual DG concentration $q$ is typically much smaller than $s_0$. 
In view of this inflow-outflow dynamics of $q$, another possible choice of an intrinsic reference value $Q_{\mathrm{ref}}$ could be a value for which inflow and outflow balance, i.e. $\dot{q}=0$ holds for a certain amount of substrate $s$, let's say initially when $s=s_0$. However, choosing $Q_{\mathrm{ref}}$ in this way would not apply to all the parameters we would like to analyse. In particular in the absence of transacylation, 
a value $Q_{\mathrm{ref}}$ such that $\dot{q}=0$ holds only exists when $V_2>V_1$, see Section~\ref{sec:QSSA}. 
Even in the cases when such reference value exists, it depends on $\sigma$, 
which makes a comparison between cases with and without transacylation, 
i.e. $\sigma=0$ cumbersome. 

We therefore prefer to set $Q_{\mathrm{ref}}= K_2$, which is independent of 
these cases and allows to normalise the Michaelis constant of the DG hydrolysis rate, thus reducing one parameter. Moreover, we introduce 
the dimensionless parameter $L = s_0/K_2$, which appears in all three terms in \eqref{eq:scaleq}.
Moreover, the dimensionless parameter 
$V = V_2/V_1$ denotes the ratio of the two MM maximal velocities. 
Finally, the nondimensional transacylation term in \eqref{eq:scaleq} features the dimensionless parameter 
${\sigma K_2^2}/{V_1}$, which compares 
DG transacylation to TG hydrolysis.
We prefer to rewrite this parameter as 
$$
\frac{\sigma K_2^2}{V_1}=\frac{\sigma K_2^2}{V_2} V =: \kappa V,
$$
where we introduce the dimensionless parameter  
$\kappa = \sigma K_2^2/V_2$. The parameter $\kappa$ compares DG transacylation to DG hydrolysis, thus the two downstream pathways for DG. 
Altogether, the rescaled model \eqref{eq:01} becomes
\begin{equation}\label{eq:02a}
\begin{cases}
\begin{aligned}
     	&\dot{s} = -m_1(s) + V \kappa q^2,\\[1mm]
    	&\dot{q} = L \left[ m_1(s) - V \left(m_2(q) + 2 \kappa q^2\right) \right], \\[1mm]
    	&\dot{p} = V \left( m_2(q) + \kappa q^2\right),\\[1mm]
	&\dot{f}  = m_1(s)+ V m_2(q), 
\end{aligned}
\end{cases}
\quad\text{where}\quad
\begin{cases}
m_1(s) = \frac{s}{K + s},\\[1mm]
m_2(q) = \frac{q}{1 + q}, \\[1mm]
K = \frac{K_1}{s_0}, L = \frac{s_0}{K_2}, \\[1mm]
V = \frac{V_2}{V_1}, \kappa = \frac{\sigma K_2^2}{V_2}. 
\end{cases}
\end{equation}
The model features three positive dimensionless parameters $K,L,V>0$ and the non-negative parameter $\kappa\ge 0$ whether transacylation is considered or not. 
Given the molecular biological background, one would expect that the initial TG concentration $s_0$ is at least not smaller than the Michaelis constants $K_2$ of DG hydrolysis, which suggests $L\ge1$
although we shall not assume this condition. 
There are also some data (see \cite{Kim-2008}) suggesting that 
$V_2$ should be somewhat larger than $V_1$, so that one can maybe expect $V\sim3$, but again we shall not assume this.

The scaled initial concentrations become
\begin{equation}\label{eq:02b} s(0) = 1, \quad  q(0) = \frac{q_0}{K_2}, \quad p(0) = 0 \quad \text{and} \quad f(0) = 0
\end{equation}
and the non-dimensional form of the conservation laws 
\eqref{eq:masscons} reads as
\begin{equation}\label{eq:02c} 3s(t) + \frac{2q(t)}{L} + p(t) + f(t) = 3  + \frac{2q(0)}{L} \quad \text{and} \quad s(t) + \frac{q(t)}{L} + p(t) = 1 + \frac{q(0)}{L} \end{equation}
which holds true for all times $t \ge 0$. Moreover, $p_{\infty} = 1 + q_0/s_0 $ and $r_{\infty} = 2+q_0/s_0$.

\subsubsection*{Phaseportrait and global a priori bounds}
The qualitative behaviour of the $s-q$ system is illustrated by phaseportrait Fig.~\ref{fig:PP-SQ}: Given initial data $s_0>0$
and assuming $q_0$ sufficiently small in comparison, the TG concentration $s$ decays due to the 
predominant TG hydrolysis. In the same process, the DG concentration $q$ builds up to a maximum value, 
before decaying alongside the breakdown of TG. We will discuss later the question 
under which conditions this coupled decay process can be approximated by a QSSA replacing the
$q$ equation. In a second regime for large initial values $q_0$ (compared to $s_0$), transacylation 
leads to a transient increase of $s$ before again $s$ and $q$ tend to zero in the large time behaviour
due to the breakdown of TG and DG.

\begin{figure}[tb]  
\centering
	\includegraphics[width=0.55\textwidth]{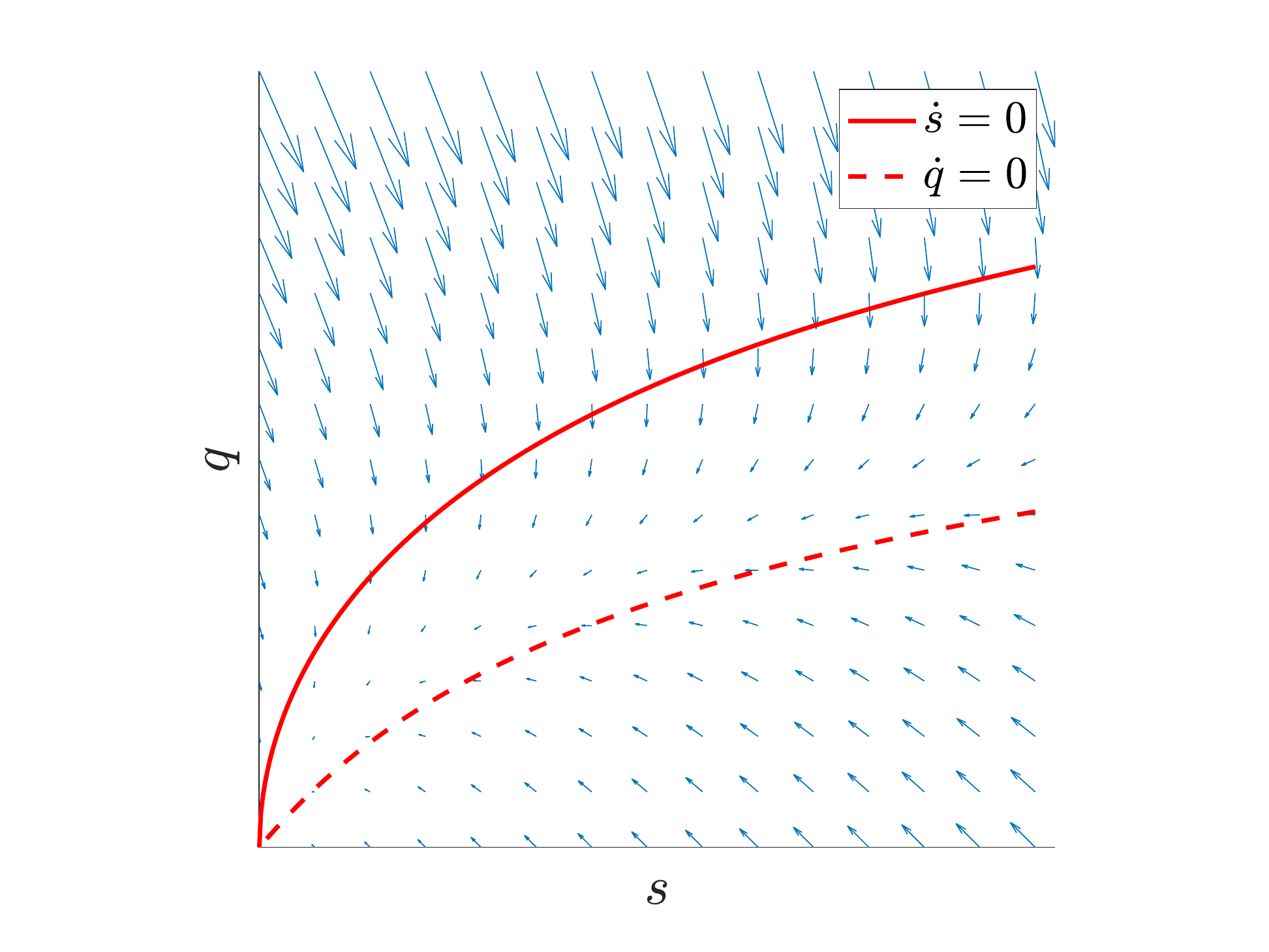} 
 	\caption[]{$s$--$q$ phase portrait; solid curve is the $s$-nullcline, dashed curve is the $q$-nullcline. 
	}
  	\label{fig:PP-SQ}
 \end{figure}

The large time behaviour of the $s-q$ system is addressed by the following lemma:

\begin{lemma}[Global bounds and exponential decay of TG/DG concentrations]\label{lem:globalbounds}\hfill\\
Consider solutions to \eqref{eq:02a} subject to initial data \eqref{eq:02b}.

Then, there exist positive constants $s_{\mathrm{max}}$, $q_{\mathrm{max}}$ and $C_1$, $C_2$ such that
\begin{equation}\label{eq:exp} 
\text{for each }t \ge 0:
\quad
\begin{cases}
0 \le s(t) \le s_{\mathrm{max}}, \quad 0 \le q(t) \le q_{\mathrm{max}},\\
0 \le s(t), q(t) \le C_1 e^{-C_2 t} 
\end{cases}
\end{equation}
If additional $q(0)$ is smaller than the value where $\dot{q}(s_{\max})=0$ holds, then 
the constant $q_{\mathrm{max}}$ can be chosen independently of $L$.
\end{lemma}
\begin{proof}
The existence of constants $s_{\max}$ and $q_{\max}$ follows directly from the mass conservation laws \eqref{eq:02c} and the non-negativity of $s$ and $q$, which imply the bounds $s(t)\le 1 + \tfrac{2q(0)}{3L} =:s_{\max}$ and $q(t) \le L + q(0)$.
While $s_{\max}$ is uniformly bounded in $L\ge1$  this is not true for this bound on $q$. However, the bound 
$s_{\max}$ allows to estimate 
\begin{equation*}
 \dot{q} = L \left[ m_1(s) - V \left(m_2(q) + 2 \kappa q^2\right) \right] \le  L \left[ m_1(s_{\max}) - V \left(m_2(q) + 2 \kappa q^2\right) \right].
\end{equation*}
Denote by $\tilde{q}(s)$ the value of $q$, where $\dot{q}=0$ holds (see Fig. \ref{fig:PP-SQ}). 
Then, $\tilde{q}(s_{\max})$ constitutes an upper solution to $q$ provided that 
$q(0)\le \tilde{q}(s_{\max})$. Hence, under this additional 
assumption, we deduce $q(t)\le  \tilde{q}(s_{\max})=:q_{\max}$, where $q_{\max}$ is now independent of $L$ (and also uniform in large $V$ and $\kappa$, cf. Lemma \ref{lem:tildeq} below).

Next, we estimate 
\[ m_1(s) = \frac{s}{K+s} \ge \frac{s}{K+s_{\max}} \quad \text{and} \quad m_2(q) = \frac{q}{1+q} \ge \frac{q}{1+q_{\max}}. \]
Let $\alpha$ and $\beta$ be arbitrary numbers such that $0 < \beta < \alpha \le 2 \beta$. For instance, we can choose 
$\alpha=1=2\beta$.
Then,
\[ \begin{aligned} (\alpha L s + \beta q)^{\displaystyle \cdot} & = -(\alpha - \beta) L m_1(s) - \beta L V m_2(q) - (2 \beta - \alpha) L V \kappa q^2 \\
& \le  - \frac{\alpha - \beta}{\alpha(K+s_{\max})} \alpha L s -  \frac{L V}{1+q_{\max}} \beta q \\
& \le - \min \left\{ \frac{\alpha - \beta}{\alpha(K+s_{\max})}, \frac{L V}{1+q_{\max}}\right\} \left( \alpha L s + \beta q\right) 
 =: -C_2 \left( \alpha L s + \beta q\right),
\end{aligned} \]
where we set $C_2 = \min
\left\{ \frac{\alpha - \beta}{\alpha(K+s_{\max})}, \frac{L V}{1+q_{\max}}\right\}$. It follows from the Gronwall inequality that
\[ \alpha L s(t) + \beta q(t) \le (\alpha L + \beta q(0)) e^{-C_2 t}. \]
We set $C_1 = \max \left\{ 1 + \frac{\beta}{\alpha L}q(0), \frac{\alpha L}{\beta} + q(0)\right\} =  \max \left\{ 1, \frac{\alpha L}{\beta} \right\}\left( 1 + \frac{\beta}{\alpha} \frac{q_0}{s_0} \right)$ and complete the proof.
\end{proof}


\section{The qualitative role of DG transacylation}\label{sec:role}

ATGL is the main TG hydrolase and thus believed to be the key regulator for 
the lipolytic cascade TG $\to$ DG $\to$ MG.   
This poses  the question if the ability of AGTL to additionally facilitate DG transacylation 
is of physiological relevance or not?

The question of the qualitative role of DG transacylation to our model systems can be viewed from three different perspectives: 
\begin{itemize}
\item[i)] How does DG transacylation effect the processing of the substrate TG?
\item[ii)] How does DG transacylation effect the production of the product MG?
\item[iii)] How does DG transacylation effect the production of the product FA?
\end{itemize}

\begin{figure}[!htb]
\centering
	\includegraphics[width=1\textwidth]{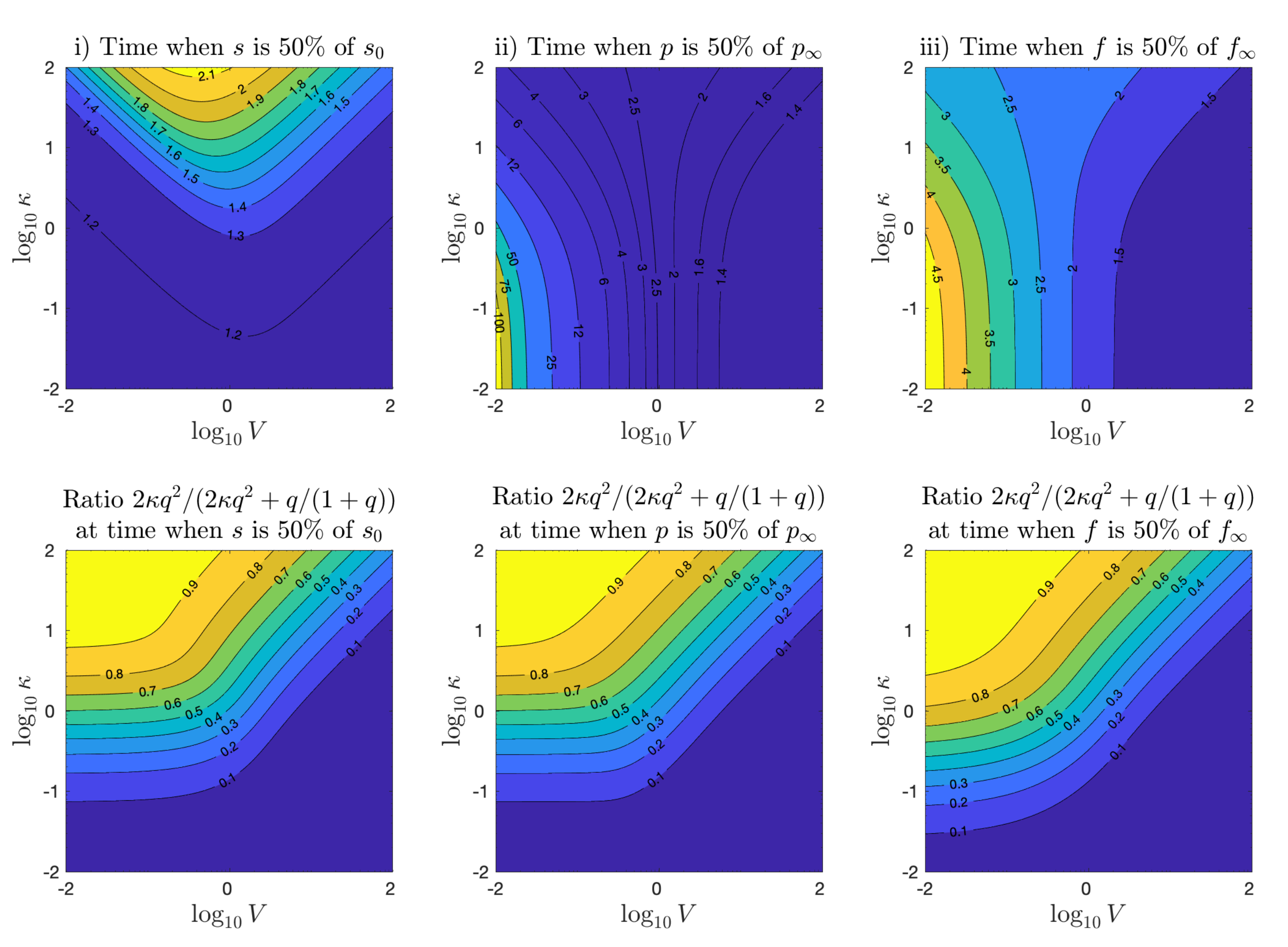}
	\caption[]{\textit{Top row:} Absolute simulation time of system \eqref{eq:02a} needed i) to process 50\% of the the initial substrate $s_0$, ii) to produce 50\% of the final products $p_{\infty}$ and iii) to produce 50\% of $r_{\infty}$. \textit{Bottom row:} Relative contribution of DG transacylation $2\kappa q^2$ in the processing of DGs at the times as the above plots. The other parameters are $L= K = 1$.}
	\label{fig:50proa}
\end{figure}

In the following, we present numerical simulations of the non-dimensionalised system \eqref{eq:02a} with MATLAB version 2021b and its built-in ode23s routines. 
In all plots, we consider initial data consisting only of TG, i.e. $s_0=1$ and $q_0=p_0=f_0=0$ and fix the parameters $L=K=1$.  

Fig.~\ref{fig:50proa} plots in the top row the simulation times, which are needed -- depending on values of the parameters  $V$ and $\kappa$ -- i) to process 50\% of the initial substrate TG $s_0$, ii) to produce 50\% 
of the final amount of MG, $p_{\infty}=s_0=1$, and iii) to produce 50\% of the final amount of FA, $f_{\infty}=2 s_0=2$. The axes plot the decimal 
$\log$ of the parameter values $(V,\kappa)$. Note that the times ii) and iii) to produce 50\% of 
$p_{\infty}$ and $f_{\infty}$ will always be larger than time i) to process 
50\% of the initial substrate, which is clearly a requirement for both other times to be reached. The time to produce 50\% of  $p_{\infty}$ can be much larger than the one of 50\% of 
$f_{\infty}$, since TG hydrolysis can reach the FA target on its own, while 
in order to reach 50\% of  $p_{\infty}$, either DG transacylation or DG hydrolysis is required.

The interpretation of Fig.~\ref{fig:50proa} is far from trivial. For a better understanding we provide the second row in 
Fig.~\ref{fig:50proa}, which plots the relative fraction of the DG transacylation rate to the total rate of DG processing, i.e. $2\kappa q^2 / (2\kappa q^2 + \frac{q}{1+q})$ in the same three cases, at the same times and under the same parameters as the top row.
Note that thanks to our choice of non-dimensional parameters in Section~\ref{sec:scaling}, the ratio $2\kappa q^2 / (2\kappa q^2 + \frac{q}{1+q})$ only involves the parameter $\kappa$ yet also depends via the solution $q$ on $V$. 
All the images in the second row of
Fig.~\ref{fig:50proa} show that DG transacylation is dominant at a level of 90\% on the left upper corner. On the other hand, DG hydrolysis is dominant
at a level of 90\% at the bottom and the extended right lower corner. 
In the in between area, the contour levels decay horizontally for small levels of $V$ and diagonally for larger values of $V$. The reason for this kink in the contour levels is the change in the values of $q$ in calculating the reaction rates $2\kappa q^2$ and $\frac{q}{1+q}$, respectively. 

Returning to the interpretation of the top row of Fig.~\ref{fig:50proa}, we see in i) that increasing the value of $\kappa$ in DG transacylation slows down the processing of the TG substrate. The slowing down effect is the strongest for values of $V\sim 1$.

For larger value of $V$, i.e. in the right upper part, the contour lines of equal times follow the contour lines of equal DG fraction in the picture below. Accordingly, the increased TG processing times corresponds in the right upper part to the increase of the dominance of   
DG transacylation, which provides increasing feedback into the TG pool.  

For small values of $V$, however, the contour lines in Fig.~\ref{fig:50proa} i) do not correspond to the ones in the image below, which means 
that structure of the DG transacylation feedback changes: In this upper left corner, where DG transacylation is overall dominant and roughly constant, the contour lines in Fig.~\ref{fig:50proa} i) follow the lines where the parameter product 
$V\kappa$ is constant, i.e. $\log_{10} \kappa = C - \log_{10} V$ for some constant $C$.
In this part of the image, the increase of the product $V\kappa$ determines the increase in feedback to the TG pool and the increase in the delay of the 50\% TG processing time. Hence, we infer that in this part of the picture, changes in $V$ and $\kappa$ do not substantially change the values of $q$. Rather, the increase of the source term $V\kappa q^2$ in the $s$ equation is mainly due to the change in the parameter product $V \kappa$. This is different to the right upper area where the contour lines of Fig.~\ref{fig:50proa} i) follow the contour lines of the image below!


Figs.~\ref{fig:50proa} ii) shows for small values of $\kappa$ vertical contour levels  as the time needed to produce 50\% of MG depends under this conditions entirely on DG hydrolysis and the parameter $V$. Such vertical contour levels for small values of $\kappa$ are also shown in Figs.~\ref{fig:50proa} iii) since DG hydrolysis also contributes to the production of FAs, yet with a much more moderate 
variation of the values compared to Figs.~\ref{fig:50proa} ii) since TG hydrolysis is a constant source of FA production. 
In both images Figs.~\ref{fig:50proa} ii) and iii), the influence of DG transacylation 
tilts the contour lines away from the middle for increasing values of $\kappa$. In the left upper half of Fig.~\ref{fig:50proa} ii), the contour lines follow again the lines 
$\log_{10} \kappa = C - \log_{10} V$ for some constant $C$, which shows 
how DG transacylation helps in the case of lacking DG hydrolysis to produce MGs and thus lower the time of 50\% $p_{\infty}$ production. The left upper half of Fig.~\ref{fig:50proa} iii) shows similar contour lines, since DG transacylation provides 
feedback into the TG pool, from which TG hydrolysis can release additional FAs. However, this DG feedback effects are somewhat more blurred compared Figs.~\ref{fig:50proa} ii) due to the ongoing TG hydrolysis, which occurs anyway.

In the right upper 
half of Fig.~\ref{fig:50proa} ii), the contour lines turn to follow the contour lines in the image below and the production of MGs  is actually slowed down by increasing $\kappa$ and thus DG transacylation. At the first glance, this seems counter intuitive since DG transacylation provides an additional pathway to produce MGs. 
However, we recall that DG hydrolysis yields one MG out of one DG while DG transacylation results in only 1/2 MG. Hence, increased DG transacylation can hinder the twice more efficient production of MGs via DG hydrolysis, which causes the production of MGs to effectively slow down. The left upper half of Fig.~\ref{fig:50proa} iii) shows again similar contour lines and a similar behaviour as DG transacylation hinders the production of FAs from DG hydrolysis. 

\subsubsection*{Relative time changes compared to dynamics without DG transacylation}

\begin{figure}[!htb]
\centering
	\includegraphics[width=1\textwidth]{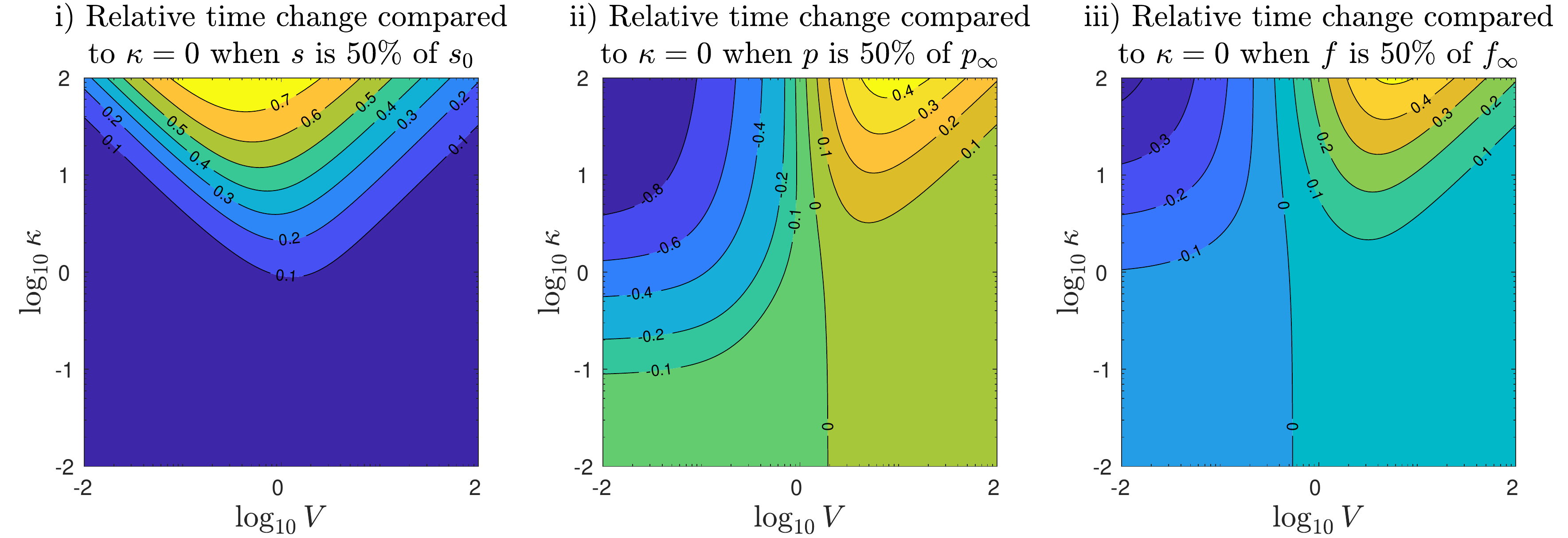}
	\caption[]{Relative time change $(t_{\kappa > 0} - t_{\kappa=0})/t_{\kappa=0}$ of the simulation time of system \eqref{eq:02a} compared the same dynamics without DG transacylation, i.e. $\kappa=0$. The left image shows the relative time change for different values of the parameter $V$ and $\kappa$ at the time when 50\% of the the initial substrate $s_0$ is processed. The middle image shows that the relative time change in order to produce 50\% of the final product $p_{\infty}$
can be either negative or positive, representing cases where DG transacylation compensates for lacking DG hydrolysis 
or slows down MG production by substrate competition with DG hydrolysis. The right picture shows that also the relative time change in producing 50\% of the final FAs $f_{\infty}$ can be negative or positive, representing cases when the feedback of DG transacylation into the TG pool allows TG hydrolysis to release additional FAs and when DG transacylation shows substrate competition 
with DG hydrolysis and thus slows down the release of FAs. The other parameters are $L= K = 1$.}
	\label{fig:50prob}
\end{figure}

Fig.~\ref{fig:50prob} shows relative time changes compared to the dynamics without transacylation, i.e. 
\begin{equation*}
\frac{t_{\kappa > 0} - t_{\kappa=0}}{t_{\kappa=0}}, 
\end{equation*}
where $t_{\kappa > 0}$ denotes the simulation times like in the cases i), ii), and iii) of Fig.~\ref{fig:50proa} for positive values of $\kappa$ and $t_{\kappa}=0$ denotes the 
corresponding simulation time without DG transacylation (yet all other parameters being the same). 

On the left, Fig.~\ref{fig:50prob} i) plots the relative time change 
with respect to processing 50\% of the initial substrate TG $s_0$. 
For small values of $\kappa$, the image show small, positive relative time changes, which corresponds to the small, positive feedback of DG transacylation into the pool of TG delaying the time to 
have processed 50\%. 
With increasing $\kappa$, the delay effect becomes more and more significant reaching up to 70\% for values of $V\lesssim1$. 
The contour lines of Fig.~\ref{fig:50prob} i) correspond to Fig.~\ref{fig:50proa} i), which is an easy consequence of the fact that for $\kappa=0$, the processing time of 50\% TG 
is due to TG hydrolysis alone and thus the same for all values of $V$. 

The middle image of Fig.~\ref{fig:50prob} ii) shows the complexity of the interplay DG transacylation vs. DG hydrolysis. 
For small values of $V$, increasing $\kappa$ leads 
to DG transacylation becoming dominant and therefore speeding 
up MG production in Fig.~\ref{fig:50prob} ii) up to 80\% and higher.
Here, DG transacylation acts as a by-pass process in order to  
produce MGs despite lacking DG hydrolysis. 
A similar speed is shown in Fig.~\ref{fig:50prob} iii), where 
DG transacylation feeds back into the TG pool, from which 
TG hydrolysis can release additional FAs. 
In contrast, for larger values of $V\gtrsim1$, Fig.~\ref{fig:50prob} ii) shows 
that increasing $\kappa$ and DG transacylation results in a substrate competition to 
DG hydrolysis, which is nevertheless the more efficient process since it produces 1 MG out of 1 DG, while DG transacylation yields on 1/2 MG. Therefore, the 50\% MG production times are slowed down by 40\% and more. 
This is again similar in Fig.~\ref{fig:50prob} iii), where DG transacylation 
hinders the immediate production of a FA by DG hydrolysis (even if TG hydrolysis can later release 
1/2 FA form the 1/2 TG produced by DG transacylation). 

\subsubsection*{Effects of DG transacylation depend on the state of the hydrolytic machinery}

We point out the effects of DG transacylation also depend on the current stage within the time evolution of the hydrolytic machinery. Fig.~\ref{fig:50proc} plots the relative time changes comparing system \eqref{eq:02a} with $L=K=1$ and $\kappa=10$ to the same system with $\kappa=0$ as function of $V$ and for different percentages $x\%$ of TG process and the corresponding percentages $(100-x)\%$ of MG and FA production. 

Fig.~\ref{fig:50proc} i) plots the slowdown in the TG substrate processing times. For $V\lesssim1$, i.e. under conditions where DG hydrolysis is lacking, the slowdown of TG process is largest for small values of $x$, i.e. at the beginning of hydrolysis, when DG concentrations are largest and 
DG transacylation provides therefore the largest feedback into the TG pool.  
For $V\gtrsim1$, i.e. under conditions where DG hydrolysis is present, the
effects of DG transacylation and thus the slowdown get smaller and 
almost independent of $x$ for large values of $V$. 

Fig.~\ref{fig:50proc} ii) shows for $V\ll1$ a strong 80\% acceleration of MG production thanks to the MGs produced by DG transacylation, which serves under this conditions as a by-pass of the lacking  
DG hydrolysis. In contrast, for values $V\gtrsim 1$, Fig.~\ref{fig:50proc} ii) displays that MG production is slowed down 
for smallish $x=10,25,\ldots$, i.e. at the beginning of hydrolysis. Yet, later in the evolution, when $x=75,90$, the early slow down is not only compensated but MG production is altogether accelerated due to DG transacylation. This example underlines the nonlinear behaviour of the 
effects of DG transacylation and how they depend on the current state of the hydrolytic machinery. This is important in view of biochemical experiments, in particular when comparing in vitro cell assays experiments to in vivo experiments, where effective DG concentrations are very difficult to determine. 

Fig.~\ref{fig:50proc} iii) shows for values $V\ll1$ and smallish $x=10,25$ a very strong speed up of 80\%
in the FA production time since DG transacylation creates additional TGs, which serve as a substrate for TG hydrolysis and therefore produce  additional FAs. However, later in the evolution for $x>50\%$, there is no noticeable  
speed up anymore since only 50\% of the final FAs come from TG hydrolysis (and in-between DG transacylation).
The remaining x-50\% of the final FAs can only be produced via DG hydrolysis, which is just as slow as it is.  

On the other hand, Fig.~\ref{fig:50proc} iii) shows for values $V\gtrsim1$ a slow down effect on FA production from DG transacylation, as it competes 
with DG hydrolysis for substrate and (after a TG hydrolysis step) yields only half of the FA release compared to DG hydrolysis.

\begin{figure}[tb]
\centering
	\includegraphics[width=1\textwidth]{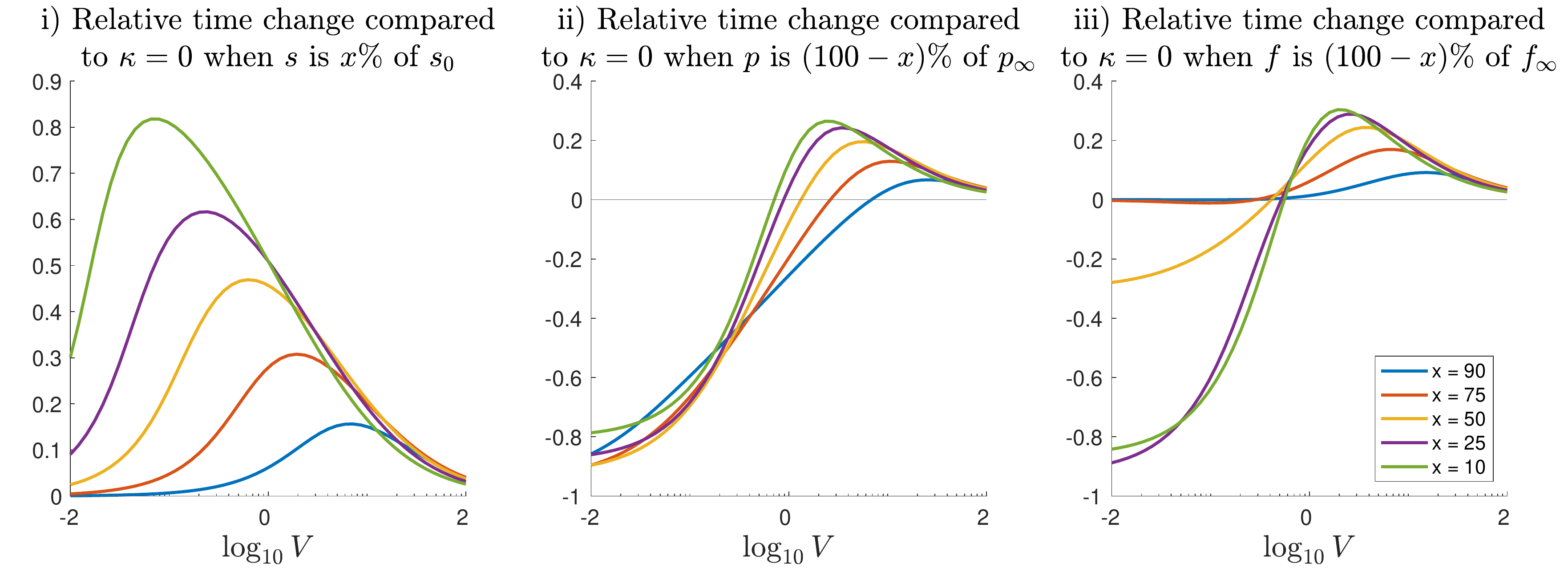}
	\caption[]{Relative time change $(t_{\kappa > 0} - t_{\kappa=0})/t_{\kappa=0}$ of the simulation time of system \eqref{eq:02a} with $L=K=1$ and $\kappa=10$ compared the same dynamics without DG transacylation $\kappa=0$. The left image i) plots the relative time changes for different percentages $x$ of TG processing as functions of the parameter $V$.
	The middle and the right images ii) and iii) plot the relative time changes for the corresponding percentages $(100-x)$ of MG and FA production as functions of the parameter $V$.	}
	\label{fig:50proc}
\end{figure}

\section{Model reduction and QSSA prerequisites} \label{sec:QSSA}
The most important equations in system \eqref{eq:02a} 
is the closed subsystems for TG and DG concentrations:
\begin{align}
     	\dot{s} &= -m_1(s) + V \kappa q^2, && s(0)=1, \label{eq:qs} \\[1mm]
    	\dot{q} &= L \bigl[ m_1(s) - V d(q)  \bigr],  &&d(q) := m_2(q) + 2 \kappa q^2.
	\label{eq:qq}
\end{align}
Here, we have scaled (w.l.o.g.) the initial TG concentration to one and introduce the function $d(q)$ as a short hand for the sum of the DG hydrolysis and DG transacylation reaction rates in eq. \eqref{eq:qq}. 
\smallskip

In this section, we investigate under which conditions system 
\eqref{eq:qs}--\eqref{eq:qq} can be reduced to a simpler model system by means of a quasi-steady-state-approximation (QSSA). QSSAs exploit that some quantities change slowly in time compared to typical time scales of the overall dynamics. Then, a zero order QSSA sets the time derivative of these quantities to zero, thus treating it quasi like for a steady state. From the viewpoint of asymptotic analysis, QSSAs are singularly perturbed ODE problems, see e.g.  \cite{Holmes-1995,kuehn2015} and the many references therein. 
Hence, approximating slowly changing quantities leads to initial layers (except in the special case where the initial data of the slowly changing quantities already matches their QSSA values), which can be resolved, for instance, by matched asymptotics. 
More accurate approximations can be obtained by calculating higher order correction terms. Note that \eqref{eq:qs} can't qualify for a QSSA since the positive feedback provided by DG transacylation is at most half of TG hydrolysis. Even if DG hydrolysis could be neglected, then the feedback DG transacylation and subsequent TG hydrolysis provides is only like TG $\to$ DG $\to $1/2 TG $\to$ 1/2 DG \dots, see Figure~\ref{fig:reactionnetwork}.
\smallskip

However, eq. \eqref{eq:qq} features a gain and two loss terms (TG hydrolysis vs. DG hydrolysis and DG transacylation), where the loss terms 
are monotone increasing in $q$. Hence, the dynamics of \eqref{eq:qq} tends to balance gain and loss terms. If this happens sufficiently fast, a QSSA is expected to serve as a simplified approximative model replacing 
\eqref{eq:qq} by an algebraic equation (and possible higher order correction terms).

As a first example, we divide eq. \eqref{eq:qq} by $L$ and assume that $L$ is a large parameter. Then, the  small parameter $\eps={L}^{-1}$  
suggests a QSSA $\eps \dot{q} \simeq 0$ whenever $L$ is sufficiently large. In the following, we will see that also the parameters 
$V$ and $\kappa$ yield QSSAs provided they are sufficiently large. In these two cases, one rescales $q \to \eps q$ in 
\eqref{eq:qq} with either $\eps = V^{-1}$ or $\eps = \kappa^{-1/2}$, which 
lead again to a left hand side $\eps \dot q\simeq 0$.

In all three cases, system \eqref{eq:qs}--\eqref{eq:qq} can be  approximated by a zero order QSSA model system by setting $\eps \dot q= 0$ and replacing the differential equation \eqref{eq:qq} by the algebraic equation
\begin{equation}\label{QSSA}
\frac{m_1(s)}{V} =  d(\tilde q) =2 \kappa \tilde{q}^2+ m_2(\tilde q), \qquad\text{with}\quad m_2(\tilde q) = \frac{\tilde q}{1+ \tilde q}, 
\end{equation}
where the function $\tilde{q}= \tilde{q}(s)$ is the unique positive solution to \eqref{QSSA} (recall the monotonicity in $\tilde{q}$) as long as it exists. In particular, $\tilde{q}$ is not defined when $\kappa=0$ and $\frac{m_1(s)}{V}\ge1$ which contradicts the bound $m_2(\tilde q)<1$
for all $\tilde q \in \mathbb{R}_+$. 
Note, that the explicit solution to the third order polynomial \eqref{QSSA} by Cardano's formula is of little practical use
and that we will make use of $\tilde{q}$ either implicitly or approximatively.  


The following Lemma \ref{lem:tildeq} states a parameter condition which imply $\tilde q \le 1/2$ and provides  a useful approximation formula of $\tilde q$ in this case. 
Note that in practical examples, we expect QSSA levels $\tilde q$ to be rather small, in particular if $V$ and/or $\kappa$ are large enough. 

\begin{lemma}[Condition ensuring $\tilde{q}\le \frac{1}{2}$ and explicit approximation formula]\label{lem:tildeq}\hfill \\ 
Assume that the parameters $(V,\kappa)$ satisfy
the condition
\begin{equation}\label{Vkappacon}
V \ge \frac{4}{1+2\kappa} \qquad \iff \qquad 
\kappa \ge \frac{2}{V} - \frac{1}{2}.
\end{equation}
For instance, suppose $V\ge4$ and $\kappa\ge0$.  

Then, $\tilde q \le \frac{1}{2}$ holds for all $s\le 1$ and $K>0$ and we can approximate  (within 10\%) 
\begin{equation}\label{QSSAapprox}
 \tilde{q} \sim \frac{1}{V}\frac{2}{\sqrt{1 + 4 (2\kappa -1)/V} +1}
= \begin{cases}
O(\frac{1}{V}) & \text{if } V \gg \kappa,\\
O(\frac{1}{\sqrt{2\kappa V})} & \text{if } V \ll \kappa.
\end{cases}
\end{equation} 
\end{lemma}
\begin{proof}
Since $d(\tilde q)$ is strictly monotone increasing in $\tilde q$, the QSSA equation \eqref{QSSA}, i.e. 
\begin{equation*} 
d(\tilde q) = 2 \kappa \tilde{q}^2 + \frac{\tilde{q}}{1+ \tilde{q}}= \frac{m_1(s)}{V}=:I.
\end{equation*} 
has a unique positive solution $\tilde q$ for every input value $I>0$. 
The explicit solution by Cardano's formula  is too cumbersome for practical use. Instead, we derive an approximation formula for sufficiently small $\tilde{q}$, i.e.   inputs 
$I$. In particular, since $d(\tilde{q})$ is monotone increasing, a sufficient condition for $\tilde{q}\le \frac{1}{2}$ is obtained from estimating $\frac{m_1}{V}\le \frac{1}{V}$ for all $s\ge0$ and $K>0$: 
\begin{equation*}
\frac{1}{V} \le d\Bigl(\tilde{q}=\frac{1}{2}\Bigr) = \frac{\kappa}{2} + \frac{1}{3}, 
\qquad \iff \qquad  \kappa \ge \frac{2}{V} - \frac{2}{3}.
\end{equation*}
Next, for sufficiently small $\tilde{q}$ (to be determined), we can rewrite 
\begin{equation}\label{QSSAApprox} 
d(\tilde{q}) = 2 \kappa \tilde{q}^2 + \tilde{q} -\tilde{q}^2 + \underbrace{\frac{\tilde{q}^3}{1+\tilde{q}}}_{=\epsilon=O( \tilde{q}^3)}=I
\end{equation} 
and the last term on the left hand side will be assumed of lower order. A formal asymptotic expansion $\tilde{q}=\tilde{q}_0+\epsilon \tilde{q}_1$ yields 
\begin{equation*}
\tilde{q}_0 = \frac{\sqrt{1 + 4 I (2\kappa -1)} -1}{2(2\kappa-1)}= 
\frac{2I}{\sqrt{1 + 4 I (2\kappa -1)} +1}, \qquad
\tilde{q}_1 = - \frac{1}{\sqrt{1 + 4 I (2\kappa -1)}}. 
\end{equation*}

In order to ensure the approximation $\tilde{q}_0$ to be real  
in the case
when $0\le2 \kappa<1$,  
it is sufficient to impose the constraint
\begin{equation}\label{VVVVV}
\frac{1}{V}\le \frac{1}{4}\frac{1}{1-2\kappa}\in\Bigl[\frac{1}{4},\infty\Bigr)\quad\text{for}\quad \kappa\in\Bigl[0,\frac{1}{2}\Bigr).
\end{equation}
From \eqref{QSSAApprox}, we calculate that 
$$
\tilde{q}_0\le \frac{1}{2} \quad\Longrightarrow\quad 
\epsilon(\tilde{q}_0) = \frac{\tilde{q}_0^3}{1+\tilde{q}_0} \le \frac{1}{12}.
$$
Thus, $\epsilon(\tilde{q})=\epsilon(\tilde{q}_0)+O(\epsilon)$ is 
 (at least) of order $O(10^{-1})$. On the other hand, straightforward calculation shows that $\tilde{q}_0\le \frac{1}{2}$ is equivalent to $I\le (1+2\kappa)/4$. Since $m_1(s)\le \tfrac{1}{V}$, a sufficient condition independent of $s$ and $K$ is 
$$
\frac{1}{V}\le \frac{1+2\kappa}{4} \quad\Longrightarrow\quad \tilde{q}_0\le \frac{1}{2},
$$
which is \eqref{Vkappacon} and also implies \eqref{VVVVV} and comletes the proof.
%
%
\end{proof}

%
%

\subsection{Conditions for QSSA} \label{sec:qssaconditions}
A QSSA which approximates $\eps\dot q\simeq 0$ is expected to feature an initial layer 
whenever the initial data $q(0)$ do not equal the initial values of the approximation $\tilde q (s(0))$, see e.g. \cite{Holmes-1995}. Indeed, we expect 
$\dot{q} \sim O(\eps^{-1})$ during the initial layer except for initial data $q(0)$, which are already $O(\eps)$-close to 
$\tilde q (s(0))$. 

Here, it is natural to consider initial data $q(0)=0$, i.e. that there is initially no DG, which implies that there is an initial layer behaviour.  
(Other initial data which are not  $O(\eps)$-close to $\tilde q(s(0))$ can be treated the same.) 
For a QSSA to be a useful and valid approximation, the initial layer must be short compared to the dynamics which is of 
interest.  

In the following, we derive necessary conditions on the parameters as well as bounds on the shortness of the 
initial layer. 
The proximity of $q(t)$ to $\tilde{q}(s(t))$
is -- in our opinion -- best studied after rewriting \eqref{eq:qq} as perturbation $\pi$ around the QSSA variable $\tilde{q}(s(t))$, i.e. 
\begin{equation}
q(t) = \tilde{q}(s(t)) + \pi(t).
\end{equation}
By using the identities $q^2 = \tilde{q}^2 + \pi (2 \tilde{q} + \pi)$ and 
$m_2(q) = m_2(\tilde{q}) + \frac{\pi}{(1+\tilde{q})(1+\tilde{q}+\pi)}$, 
we rephrase the right hand side of \eqref{eq:qq} as 
\begin{equation}\label{eq:qpi}
 \dot{q} = - L V \pi \underbrace{\left[ 2 \kappa (2 \tilde{q} + \pi) +\frac{1}{(1+\tilde{q})(1+\tilde{q}+\pi)}\right]}_{>0}. 
\end{equation}
Note that the squared bracket on the r.h.s. is always strictly positive 
(since $q = \pi +\tilde{q} \ge 0$). Hence, the equality $\mathrm{sign}(\dot{q})= - \mathrm{sign}(\pi)$ follows, which is a useful observation. 

\begin{figure}[htb]
\centering
	\includegraphics[width=1\textwidth]{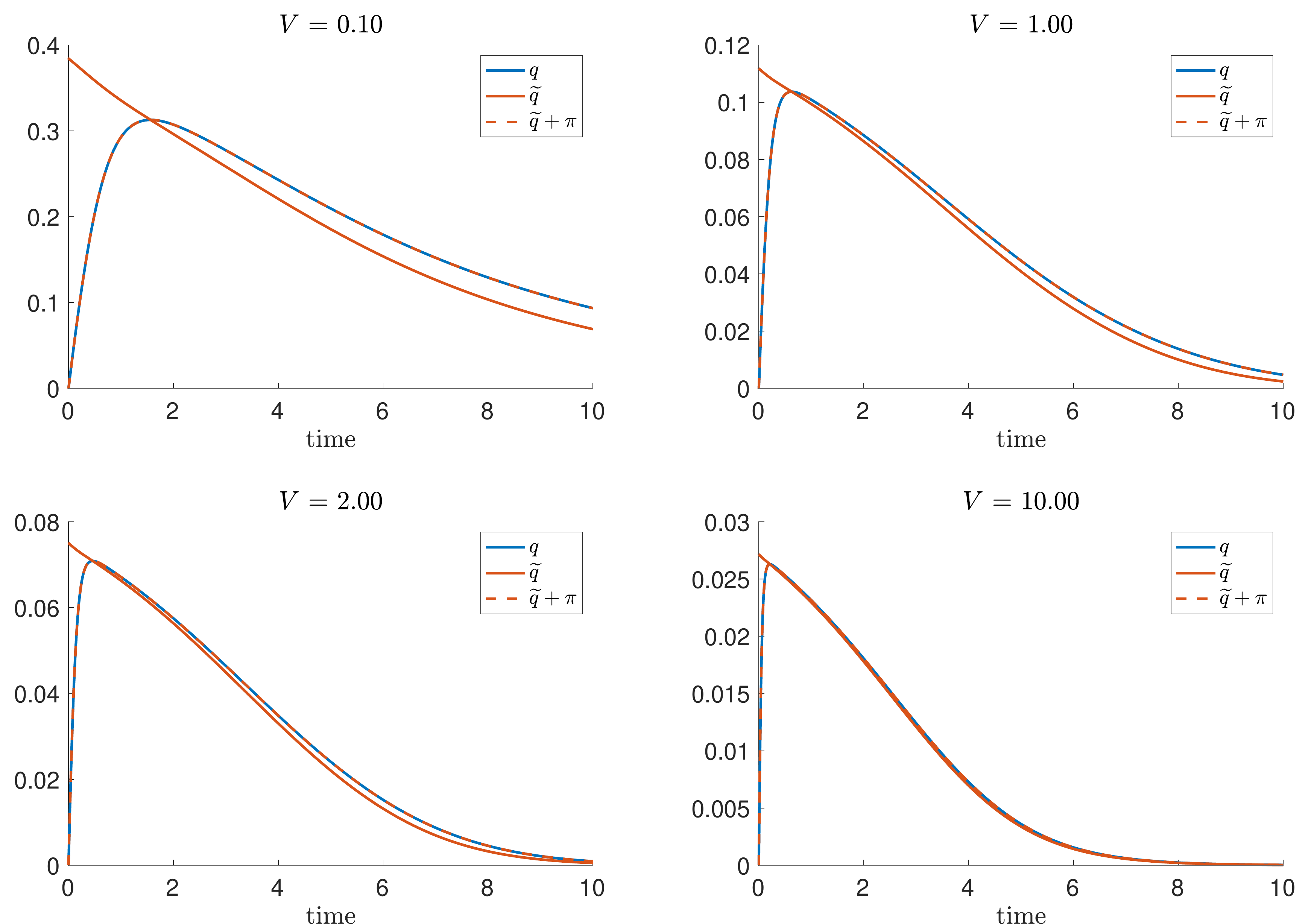}
\caption{The evolution of $q(t)=\tilde{q}(s(t))+\pi(t)$ and the formal QSSA $\tilde q(s(t))$ for $L = K = 1$, $\kappa = 16$ and different values  $V=0.1,1,2,10$. During an initial layer phase, $q$ grows from $q(0)=0$ to the maximum where $q(t_m)=\tilde q(s(t_m))$ and $\pi=0$ holds. For the subsequently decay phase, $q$ can be approximated by a QSSA $\tilde{q}(s(t))$ for sufficiently large parameters $V$. Note that $\pi(t)>0$ holds during the decay phase.}
	\label{fig:QSSA}
\end{figure}

Figure~\ref{fig:QSSA} compares the evolution of $q(t)$ from 
\eqref{eq:qs} and $\tilde q(s(t))$ for different values of $V$: Starting at the initial data without DG, i.e. $q(0)=0$ resp. $\pi(0)=-\tilde{q}(s(0))$, TG hydrolysis yields 
$s(t)$ to continuously decreases from it's starting value $s(0)=1$
(not plotted) and $q(t)$ to grow towards a maximal value $q(t_m)$ at a time $t_m>0$ when $\dot q(t_m)=0$ holds. Thus \eqref{eq:qpi} implies $\pi(t_m)=0$ and $q(t_m)=\tilde q(s(t_m))$.
At time $t_m$ of maximal DG concentration, 
the rates of DG hydrolysis and DG transacylation balance exactly TG hydrolysis. 

For times larger than $t_m$, the DG concentration $q(t)$ decays. Hence, \eqref{eq:qpi} implies $\pi(t)>0$ for $t>t_m$, 
which can be interpreted that 
the decay of the rates of DG hydrolysis and transacylation lags behind 
the decay of the TG hydrolysis rate
caused by the decay of $s(t)$, 
see also phaseportrait Fig.~\ref{fig:PP-SQ}. 
\smallskip

Altogether, Figure~\ref{fig:QSSA} suggests that a QSSA can be used provided that the initial layer interval $[0,t_m]$ is sufficiently short and that after $t_m$ the perturbation $\pi(t)$ is sufficiently small. Figure~\ref{fig:QSSA} illustrates that we can expect a QSSA for $q$ to hold for sufficiently large values of $V$. Analog plots can be made in favour if a QSSA for sufficiently large values of $L$ and $\kappa$. 
\smallskip

We would like to estimate the duration of the initial layer in the following.  
First, we derive the evolution equation of $\pi(t)$. By taking the derivate of \eqref{QSSA}, we calculate with $\tilde q$ short for $\tilde q(s)$
\begin{equation*}
\dot{\tilde{q}} =
\dot{\tilde{q}}(s) = \frac{m_1'(s) }{V d'(\tilde{q})} \dot{s},
\end{equation*}
where we use the formulas: 
\begin{equation}\label{useful}
m_1'(s) = \frac{K}{(K+s)^2} \qquad\text{and}\qquad
d'(q) = 4 \kappa {q}+ m'_2(q)>0, \quad m_2'(q) = \frac{1}{(1+q)^2}, 
\end{equation}
By recalling $\dot q = \dot {\tilde q} + \dot \pi$, we obtain 
the evolution equation for $\pi$: 
\begin{equation}\label{eq:pi}
\dot{\pi} = - L V \pi \left[ 2 \kappa (2 \tilde{q} + \pi) +\frac{1}{(1+\tilde{q})(1+\tilde{q}+\pi)}\right] - \frac{m_1' (s)}{V d'(\tilde{q})} \dot{s},
\end{equation}
subject to the initial data $\pi(0)=-\tilde{q}(s(0))=-\tilde q(1)<0$.

The time scales of eq. \eqref{eq:pi} are key to understanding necessary conditions for a QSSA. 
The first term in \eqref{eq:pi} acts as a stabilising terms and aims  to relax $\pi(t)$ towards zero (since the square bracket is positive). 
On the other hand, the second term is always positive (as $\dot s<0$). Hence, it  
helps $q$ to reach the maximal level $\tilde{q}(s(t_m))$ faster, but afterwards becomes a destabilising factor causing $\pi>0$ as long as  $\dot {q}<0$ resp. $\dot{s}<0$. Roughly speaking, the second term quantifies the adaptation of the QSSA $\tilde q$ due to the decay of $s$.

As a first observation, we suppose $\dot{s}\ll1$ and thus the second term on the right hand side of \eqref{eq:pi} negligible. Then, linearisation of \eqref{eq:pi} (i.e. rescaling $\pi \to \eps \pi$ and letting $\eps \to 0$), yields indeed exponential stability of $\pi=0$:  
\begin{equation*}
\dot{\pi} \simeq - L V \pi \left[ 4 \kappa \tilde{q}  +\frac{1}{(1+\tilde{q})^2}\right] =   - L V d'(\tilde{q}) \,\pi, \qquad \text{in the case}\quad \dot{s}\ll1.
\end{equation*}
The eigenvalue $-L V d'(\tilde{q})<0$ characterises the linear stability of the QSSA $\tilde q(s)$ provided that the decay of $s$ is negligible.
\begin{remark}
The fast timescale of classical QSSA for a single substrate - single enzyme reaction is often calculated from the eigenvalue $L V d'(\tilde{q})$ \cite[p.~179]{Murray}, i.e.  $T = 1/(L V d'(\tilde{q}))$. In the following, we will derive a better estimation of the timescale, see \eqref{eq:tsT} below.
\end{remark}

\subsection{Time scale analysis}
First, we consider \eqref{eq:pi} subject to the initial data $\pi(0)= -\tilde{q}_0= -\tilde{q}(s(0))$, $q(0)=0$  and $\dot s(0) = -m_1(s(0))$ and calculate 
\begin{equation}\label{eq:pi0}
\dot \pi \vert_{t=0} =  L V \underbrace{\tilde{q}_0  \left[ 2 \kappa \tilde{q}_0 +\frac{1}{(1+\tilde{q}_0)}\right]}_{=m_1/V} + \frac{m_1'  m_1}{V d'(\tilde{q}_0)}  = m_1 L + \frac{m_1' m_1}{V d'(\tilde{q}_0)}, 
\qquad \tilde q_0 := \tilde q(s(0)).
\end{equation} 
Secondly, at $t=t_m$ when $\pi(t_m)= 0$, $q(t_m)=\tilde q(s(t_m))$ and  $\dot s = -m_1+ V\kappa \tilde q^2$, \eqref{eq:pi} yields 
\begin{equation}\label{eq:pitm}
\dot \pi \vert_{t=t_m} =  \frac{m_1'}{V d'(\tilde{q}_m)}\left[m_1  - V\kappa \tilde q_m^2\right] = \frac{m_1'}{V d'(\tilde{q}_m)}\left[\frac{m_1}{2} + \frac{V m_2(\tilde q_m)}{2}\right], 
\qquad \tilde q_m := \tilde q(s(t_m)),
\end{equation} 
where the last equality holds due to $\dot q(t_m)=0$, i.e. $m_1 = V d(q)$ and $q(t_m)=\tilde q_m$.

Thirdly, we consider the halfway-up timepoint when $\pi = -\frac{\tilde{q}(s)}{2}$ holds in $[0,t_m]$. Evaluating $\dot \pi$ yields
\begin{align*}
\dot \pi \vert_{\pi = -\frac{\tilde{q}(s)}{2}} &=  L V \underbrace{\left[ \frac{3}{2} \kappa \tilde{q}^2 +\frac{1}{(2+\tilde{q})}\frac{\tilde{q}}{(1+\tilde{q})}\right]}_{\frac{1}{(2+\tilde{q})} \frac{m_1}{V} \le \cdots \le \frac{3}{4}\frac{m_1}{V}\quad\ } + \frac{m_1'} {V d'(\tilde{q})}\!\!\!\underbrace{\left[{m_1} -V\kappa {q}^2\right]}_{\frac{m_1}{2}+ \frac{V m_2}{2}\le \cdots \le m_1\qquad}   
\end{align*} 
where the estimates of first term used that $m_1 = V d(\tilde{q})$ and the lower bound in the second term follows 
from $\dot q>0$ impllying $2V\kappa q^2 < m_1 - Vm_2$. 
Note also that $\frac{1}{2+\tilde{q}}\lesssim \frac{1}{2}$ for small values of $\tilde q$, which we expect for large values of $V$ and $\kappa$. Hence, we observe that 
\begin{equation*}
\dot \pi \vert_{\pi = -\frac{\tilde{q}(s)}{2}} \simeq \frac{\dot \pi \vert_{t=0}+\dot \pi \vert_{t=t_m}}{2}
\end{equation*}
We therefore take this average as a typical value of $\dot \pi$ 
throughout the initial layer and use it as an estimation $T$ of the duration $t_m$ of the initial layer:
\begin{equation} \label{eq:tsT}
\tilde q(s(t_m)) =  T \frac{\dot \pi \vert_{t=0}+\dot \pi \vert_{t=t_m}}{2} \qquad\Rightarrow\qquad
T := \frac{2 \tilde q(s(t_m))}{\dot \pi \vert_{t=0}+\dot \pi \vert_{t=t_m}}, 
\end{equation}
(and where of course $T\simeq t_m$ should hold for a good estimation $T$). 
%

In the following, we estimate $T$ in terms of the parameters. First, by recalling \eqref{QSSA}, we observe that 
\begin{equation*}
V \tilde q d'(\tilde q) = V \underbrace{\left[4 \kappa \tilde q^2 + \frac{\tilde q}{(1+\tilde q)^2}\right]}_{\frac{1}{(1+\tilde{q})} \frac{m_1}{V} \le \cdots \le 2\frac{m_1}{V}\quad\ } = \theta(\tilde q) m_1, \qquad 
\text{where} \quad \theta(\tilde q) \in \left[\frac{1}{1+\tilde q}, 2\right],\quad \forall \tilde q\ge0,
\end{equation*}
and moreover, $\theta\in[\tfrac{1}{2},2]$ for $\tilde q \le 1$. Next, we rewrite $T$ as 
\begin{equation}
T 
= \frac{1}{T_1^{-1}+T_2^{-1}+T_3^{-1}}
\end{equation}
and identify the terms $T_1$ and $T_2$ in \eqref{eq:pi0} and $T_3$ in \eqref{eq:pitm} using the notations $\tilde q_0 := \tilde q(s(0)) = \tilde q(1)$ and $q_m := \tilde q(s(t_m))$:
\begin{align}
T_1 &= \frac{2}{m_1(s(0))}\frac{\tilde q_m}{L} = 2(K+1)\frac{\tilde q_m}{L}, \label{T1} \\
T_2 &=  \frac{2 }{m_1'(s(0))} \frac{\tilde q_m}{\tilde q_0}\frac{V \tilde q_0 d'(\tilde{q}_0)}{m_1(s(0))}= \frac{2 (K+1)^2}{K} \frac{\tilde q_m}{\tilde q_0} \theta, \qquad 
\theta(\tilde q) \in\left[\frac{1}{1+\tilde q},2\right],\label{T2} \\
T_3 &=  \frac{2 }{m_1'(s(t_m))}
\frac{\tilde q_m V d'(V\tilde{q}_m)}{{m_1} -V\kappa \tilde{q}^2_m}
= \frac{2 (K+s(t_m))^2}{K} \frac{\theta}{\psi},\qquad  \psi\in[\tfrac{1}{2},1],\label{T3} 
\end{align}
where we have used for $T_3$ in \eqref{eq:pitm} that $\tilde q_m = \tilde q(t_m) = q(t_m)$ and ${m_1} -V\kappa {q}^2=\psi m_1$ for some $\psi\in[\tfrac{1}{2},1]$. 

In particular, we will estimate \eqref{T2} and \eqref{T3} as
\begin{equation*}
T_2 =  \frac{\tilde q_m}{\tilde q_0} \frac{2 (K+1)^2}{K} \theta \le \frac{4(K+1)^2}{K},\qquad 
T_3 = \frac{2 (K+s(t_m))^2}{K} \frac{\theta}{\psi} \le \frac{8(K+1)^2}{K},
\end{equation*}
since $\tilde q_m\le \tilde q_0$, $\theta\le 2$ and $\psi^{-1} \le 2$.
We then obtain the upper bound 
\begin{equation}\label{Tbound}
T \le \frac{2(K+1)^2}{\frac{L}{\tilde q_m}(K+1) + \frac{3}{4} K}.
\end{equation}
\medskip
 
The time scale $T$ has to be compared a time scale representing the decay of the substrate $s$. Since
$\frac{m_1}{2} \le -\dot s \le m_1$ holds during the initial layer (due to $\dot q>0$ implying $2V\kappa q^2 < m_1 - Vm_2$), we choose the Michaelis Menten model $\dot s = - \frac{3}{4} m_1(s)$ to provide a reference time of the decay of $s(t)$. Straightforward integration with $s(0)=1$ yields 
\begin{equation*}
K \ln(s(t)) + s(t)= 1 - \frac{3}{4} t.
\end{equation*}
As comparison for the times scale of the initial layer, we introduce the time $T_{90\%}$ when $s(T_{90\%})=\tfrac{9}{10}$, 
which means that at time $T_{90\%}$ about 90\% of the TG substrate is still to be processed. We calculate  
\begin{equation}\label{T-90}
T_{90\%} = \frac{2}{15} + \frac{4}{3} K \ln\Bigl(\frac{10}{9}\Bigr) \simeq \frac{2}{15}(1+ K).
\end{equation}

As a condition on the estimation $T$ of the duration of the initial layer, we impose 
\begin{equation}\label{TT75}
\frac{2(K+1)^2}{\frac{L}{\tilde q_m}(K+1) + \frac{3}{4} K} = T \le T_{90\%} = \frac{2}{15} (1+K) \iff
\frac{L}{\tilde q_m} \ge 15 - \frac{3}{4} \frac{K}{K+1}.
\end{equation}

We state the following Proposition:
\begin{proposition}\label{prop:QSSA}
Denote by $t_m$ the time point when $q(t_m)$ is maximal, thus satisfying $q(t_m)=\tilde q(s(t_m))=:\tilde q_m$. Assume that the (not explicit, but useful) condition
\begin{equation}\label{Initiallayershort}
\frac{L}{\tilde q_m} \ge 15 - \frac{3}{4} \frac{K}{K+1}
\end{equation}
holds. Then, $\tilde q$ can be considered suitable for a QSSA in the sense that after the estimation $T$ of the duration of the initial layer least $90\%$ of the TO substrate still needs processing. 
The weaker, yet simplified condition 
\begin{equation}\label{simple}
\tilde q_m = \tilde q(s(t_m)) \le \frac{L}{15},
\end{equation}
can serve as a useful rule of thumb, which is independent of $K$. 

Both conditions \eqref{Initiallayershort} and \eqref{simple} are satisfied either be $L$ being large or by
$\tilde q(s(t_m))$ to be sufficiently small at the end time $t_m$ of the initial layer. Smallness of $\tilde q(s(t_m))$ is implied 
from $V$ or $\kappa$ (or both) being sufficiently large, cf. \eqref{QSSA} and Lemma~\ref{lem:tildeq}.

By inserting \eqref{simple} into \eqref{QSSA} and using $m_1\le 1$,
we obtain the explicit conditions 
\begin{equation} \label{kappaV}
\kappa \ge \frac{15}{2L}\left(\frac{15}{LV} - \frac{1}{1+\frac{L}{15}}\right), \quad\text{or}\quad
V \ge \frac{1}{\frac{2\kappa L^2}{15^2} + \frac{1}{1+\frac{15}{L}}}
\end{equation}
Altogether, eqs. \eqref{simple} and \eqref{kappaV} underline that we can expect a QSSA of the $q$ equation to be applicable if $L$, $V$ and/or $\kappa$ are sufficiently large.  
\end{proposition}

\begin{remark}
Looking back to numerical example Figure~\ref{fig:QSSA}, condition \eqref{Initiallayershort} for $L=1$
requires $\tilde{q}\lesssim 0.068$. This is clearly satisfied for $V=10$ where $\tilde{q}_m\sim 0.028$, 
but not for $V=2$, where $\tilde{q}\sim 0.073$. In this case, and for $K=1$, we calculate $T_{90\%}\sim 0.26$ while \eqref{Tbound} yields 
for $V=$ that $T\lesssim 0.28$ and we see that $T\le T_{90\%}$ is not ensured. 
Nevertheless, when looking at the case $V=2$ in Figure~\ref{fig:QSSA}, one would think that 
the QSSA constitutes already a very nice approximation of $q$ after the initial layer. 
One could suggest to relax the condition $T\le T_{90\%}$, to for instance $T\le T_{84\%}$ and calculate 
analog to above the condition $\tilde{q}\lesssim 0.11$, which is satisfied for $V=2$. 
However, we will elaborate this example further in the discussion and show that one has to be careful 
with being generous in accepting a QSSA. 
\end{remark}

\begin{remark}
We point out that $T_1=2(K+1)\frac{\tilde q_m}{L}$ is crucial for deriving condition \eqref{TT75}. Without $T_1$, we obtain the contradiction 
\begin{equation*}
\frac{2(K+1)^2}{\frac{3}{4} K} \le \frac{2}{15} (1+K) \iff
1 + \frac{19}{20} K \le 0.
\end{equation*}
This is not surprising since $T_2$ and $T_3$ are contributions, which are cause by the adaptation of $\tilde{q}$ due to the decay of $s(t)$ and it would be indeed surprising if this feedback alone would enable a sufficiently short initial layer.  
\end{remark}

\section{Asymptotic expansions}\label{sec:asymptotics}
In this section, we perform asymptotic expansions of the non-dimensionalised full model system \eqref{eq:02a}, which we recall for the convenience of the reader
\begin{equation}\label{eq:full}
\begin{cases}
\begin{aligned}
     	&\dot{s} = -m_1(s) + V \kappa q^2,\\[1mm]
    	&\frac{1}{L}\dot{q} = \left[ m_1(s) - V \left(m_2(q) + 2 \kappa q^2\right) \right], \\[1mm]
    	&\dot{p} = V \left( m_2(q) + \kappa q^2\right),\\[1mm]
	&\dot{f}  = m_1(s)+ V m_2(q), 
\end{aligned}
\end{cases}
\quad\text{where}\quad
\begin{cases}
m_1(s) = \frac{s}{K + s},\\[1mm]
m_2(q) = \frac{q}{1 + q}, \\[1mm]
K = \frac{K_1}{s_0}, L = \frac{s_0}{K_2}, \\[1mm]
V = \frac{V_2}{V_1}, \kappa = \frac{\sigma K_2^2}{V_2},
\end{cases}
\end{equation}
We will consider the three cases which allow a QSSA: sufficiently large $L$, $V$ and $\kappa$. 
In the case of large $L$,  the left hand side of $q$ equation in \eqref{eq:full} features the small parameter $\eps=\tfrac{1}{L}$. By setting $\eps=0$, the $q$ equation 
reduces to the algebraic equation \eqref{QSSA}. The corresponding zero order QSSA model system reads as 
\begin{equation}\label{eq:reducedL}
\begin{cases}
\begin{aligned}
     	&\dot{\tilde{s}} = -m_1(\tilde{s}) + V \kappa \tilde{q}^2,\\[1mm]
    	&\dot{\tilde{p}} = V \left( m_2(\tilde{q}) + \kappa \tilde{q}^2\right),\\[1mm]
	&\dot{\tilde{f}}  = m_1(\tilde{s})+ V m_2(\tilde{q}), 
\end{aligned}
\end{cases}
\quad\text{where $\tilde{q}$ solves uniquely }\quad 
\frac{m_1(\tilde{s})}{V} =  2 \kappa \tilde{q}^2+ m_2(\tilde q).
\end{equation}
We point out that the QSSA $\tilde{q}(\tilde{s}(t))$ from  \eqref{eq:reducedL} is not the 
same as $\tilde{q}(s(t))$ with $s(t)$ solving the full system \eqref{eq:full}, which we have studied in the previous Section~\ref{sec:qssaconditions}. However, the feedback from $\tilde{q}(\tilde{s}(t))$ to the evolution of $\tilde{s}$ in 
\eqref{eq:reducedL} is rather weak (recall $\tfrac{m_1}{2}\le -\dot{s}\le m_1$) and was therefore neglected for the sake of clarity 
in the derivation the conditions of Proposition~\ref{prop:QSSA}, which are anyway not sharp.

In the following, we improve the the approximative quality of the zero order QSSA system \eqref{eq:reducedL} by deriving higher order corrections. We will first consider the case of large $L$ and later the cases of large $V$ and $\kappa$. 
We will see that all these cases lead to somewhat different approximative systems and explain the difference in terms of the interpretation of the parameters. In the case of large $V$, we will even provided second order corrections, since 
it has been suggested that $V\sim3$ is a realistic value, which would make second order terms necessary.

\subsection{First order asymptotic expansion for $L$ large}
The parameter $L=\tfrac{s_0}{K_2}$ represents to ratio of initial TG substrate to the MM constant of DG hydrolysis. 
One could think that $s_0$ should be significantly larger than a realistic value $K_2$, which would be in favour of considering $L$ large. However, this is not so clear since ATGL and HSL are assumed bound to the surface of lipid droplets, the TG substrate which is actually available might need to be also near the surface. So effectively, the available $s_0$ might not be  larger than $K_2$.  

With $\eps=\tfrac{1}{L}$ being a small parameter, we rewrite the $q$ equation of \eqref{eq:full} as 
\begin{equation}\label{eq:fullq}
\eps \dot{q} = m_1(s) - Vd(q), \qquad \text{where}\quad d(q) = 2 \kappa q^2 + \frac{q}{1+q}. 
\end{equation}
Inserting the asymptotic expansion $q=q_0 + \eps q_1 + O(\eps^2)$ into \eqref{eq:fullq} and comparing coefficients yields in order 
$\eps^0$ yields 
\begin{equation}\label{eq:zero}
m_1(s) = V d(q_0), \qquad \Rightarrow \quad q_0(t) = \tilde{q}(s(t)), 
\end{equation}
which is the solution of the algebraic QSSA equation \eqref{QSSA}.
In order $\eps^1$, we obtain
\begin{equation}\label{eq:first}
\dot{q}_0 = - V d'(q_0) q_1, \qquad \Rightarrow \quad
q_1(t) = -\frac{\dot{\tilde{q}}}{V d'(\tilde{q})} = 
-\frac{m_1'(s) \dot{s}}{(V d'(\tilde{q}))^2},  
\end{equation}
where we have calculated $\dot{\tilde{q}}$ from differentiating $m_1(s) = V d(\tilde{q})$. 
In order to obtain an expression of $q_1$ in terms of $s$ only, we insert the $s$ equation of \eqref{eq:full} into 
the right hand side of \eqref{eq:first}, we have with $\tilde{q}=\tilde{q}(s(t))$
\begin{equation}\label{eq:qone}
q_1 = -\frac{m_1'(s) }{(V d'(\tilde{q}))^2}\left(-m_1(s) + V \kappa q^2\right) = 
-\frac{m_1'(s) }{(V d'(\tilde{q}))^2}\left(-m_1(s) + V \kappa \tilde{q}^2\right)+ O(\eps).
\end{equation}
Hence, we have as first order QSSA 
\begin{equation}
q = \tilde{q}(s(t)) - \frac{1}{L} \frac{m_1'(s)(-m_1(s) + V \kappa \tilde{q}^2)}{(V d'(\tilde{q}))^2} + O\Bigl(\frac{1}{L^2}\Bigr).
\end{equation}

We are now able to derive the first order approximative system. 
Inserting eqs. \eqref{eq:zero} and \eqref{eq:first} into the $s$ equation of full system \eqref{eq:full} yields 
\begin{equation*}
\dot{s} = - m_1(s) + V \kappa\tilde{q}^2 - 2 \eps \tilde{q} \frac{m_1'(s) \dot{s}}{(V d'(\tilde{q}))^2} + O(\eps^2),
\end{equation*}
and, therefore 
\begin{equation}\label{eq:sfirst}
\dot{s} = \left(1- 2 \eps \tilde{q} \frac{m_1'(s)}{(V d'(\tilde{q}))^2}\right)\left( - m_1(s) + V \kappa\tilde{q}^2\right)  + O(\eps^2).
\end{equation}
Next, we insert eqs. \eqref{eq:zero} and \eqref{eq:first} into the $p$ equation of \eqref{eq:full}:
\begin{equation}\label{eq:pfirst}
\dot{p} = V \kappa{q}^2 + V \frac{q}{1+q} = V \kappa\tilde{q}^2 + V \frac{\tilde{q}}{1+\tilde{q}} 
+ \eps V \left(2 \kappa\tilde{q} + \frac{1}{(1+\tilde{q})^2}\right) q_1
+ O(\eps^2),
\end{equation}
where $q_1$ is given by \eqref{eq:qone}. 
Finally, inserting eqs. \eqref{eq:zero} and \eqref{eq:first} into the $f$ equation of \eqref{eq:full} yields
\begin{equation}\label{eq:ffirst}
\dot{f}  = m_1(s)+ V \frac{\tilde{q}}{1+\tilde{q}} + \eps V\frac{1}{(1+\tilde{q})^2} q_1 + O(\eps^2),
\end{equation}
where $q_1$ is given by \eqref{eq:qone}. Altogether, we obtain as first order QSSA model system 
\begin{equation}\label{eq:firstorderL}
\begin{cases}
\begin{aligned}
     	&\dot{s} = \left(1- 2 \eps \tilde{q} \frac{m_1'(s)}{(V d'(\tilde{q}))^2}\right)\left( - m_1(s) + V \kappa\tilde{q}^2\right),\\[0mm]
    	&\dot{p} = V \kappa\tilde{q}^2 + V \frac{\tilde{q}}{1+\tilde{q}} 
+ \eps V \left(2 \kappa\tilde{q} + \frac{1}{(1+\tilde{q})^2}\right) q_1,\\[0mm]
	&\dot{f}  = m_1(s)+ V \frac{\tilde{q}}{1+\tilde{q}} + \eps V\frac{1}{(1+\tilde{q})^2} q_1,
\end{aligned}
\end{cases}
\text{where}\
\begin{cases}
m_1(s) = \frac{s}{K + s},\\[1mm]
\tilde{q} = \tilde{q}(s(t)),\\[1 mm]
q_1 =-\frac{m_1'(s) (-m_1(s) + V \kappa \tilde{q}^2)}{(V d'(\tilde{q}))^2},\\[1mm]
\eps = \frac{1}{L}.
\end{cases}
\end{equation}

\begin{figure}[!htb]
\centering
	\includegraphics[width=0.9\textwidth]{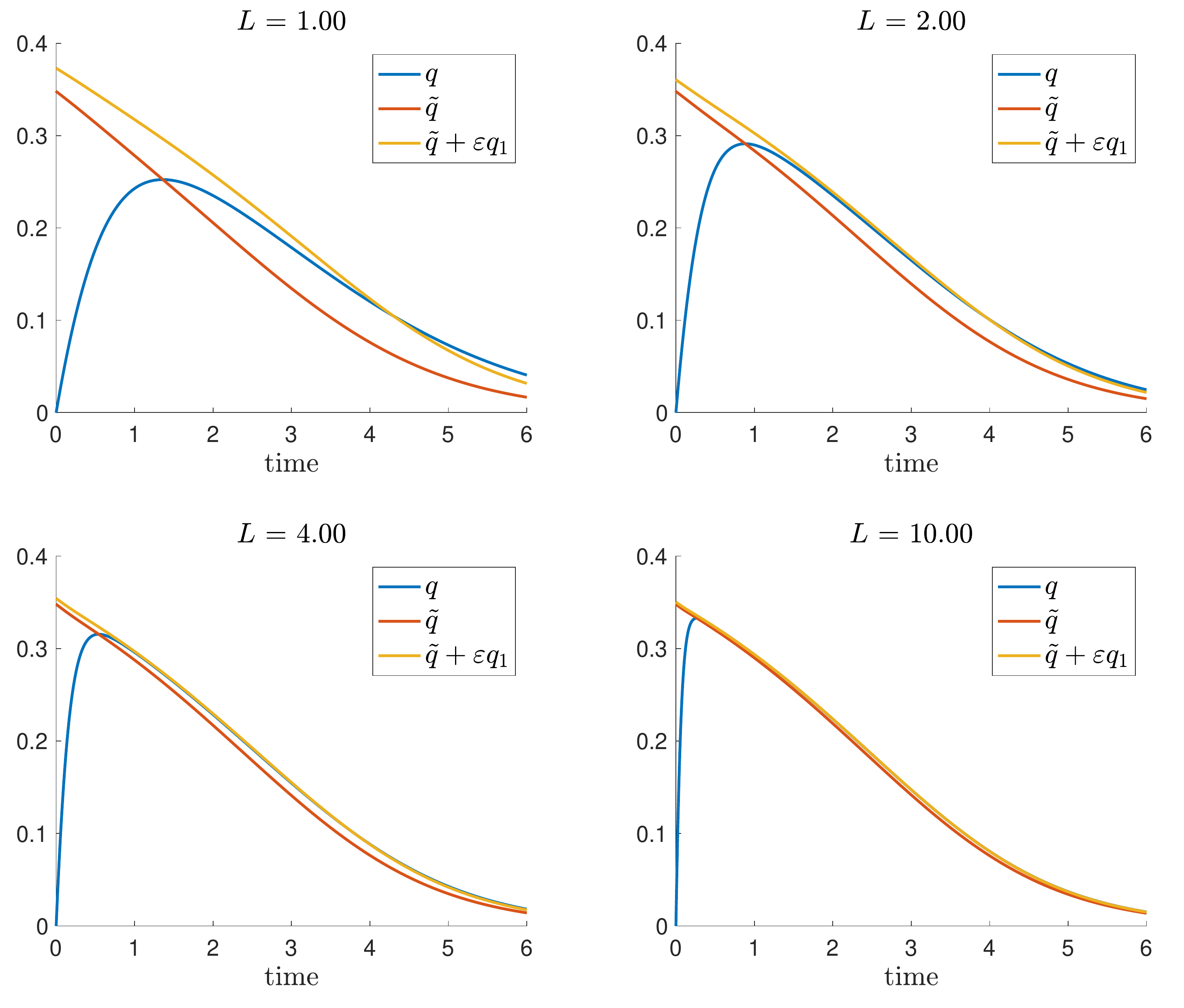}
	\caption[]{Solution $q$ of \eqref{eq:full} and zero order $\tilde{q}$ and first order QSSA  $\tilde{q}+\eps q_1$ given by \eqref{eq:firstorderL} for $L = K = \kappa = 1$, different values of $V=1,2,4,10$ and $\eps=\tfrac{1}{L}$.}
	\label{fig:reducedL}
\end{figure}

\begin{remark}
In the first order reduced model system \eqref{eq:firstorderL}, the $s$ equation is a scalar equation. 
Its solution then allows to integrate up $p$ and $f$. However, system \eqref{eq:firstorderL} depends heavily on 
$\tilde{q}(s(t))$, which is the awkward 
solution of the third order polynomial \eqref{QSSA}, which make the practical use of system \eqref{eq:sfirst}--\eqref{eq:ffirst} 
questionable.  
\end{remark}
\subsection{First and higher order asymptotic expansion for $V$ large}
We consider the case $V$ large and $V_2\gg V_1$ holds. In this parametric regime, DG hydrolysis, that is the 
loss term $m_2$ in the $q$ equation in system \eqref{eq:full}, is much faster than the inflow term $m_1$ from TG hydrolysis. Hence, $q$ is expected to be $O(\eps)$ and ought to be rescaled accordingly, i.e. $q = \eps Q$ with $\eps = 1/V$. Hence, we obtain from non-dimensionalised system \eqref{eq:full} the following rescaled model 
\begin{equation}\label{eq:fullV}
\begin{cases}
\begin{aligned}
     	&\dot{s} = -m_1(s) + \eps \kappa Q^2,\\[1mm]
    	&\eps \dot{Q} = L \bigl[ m_1(s) - \bigl(M_2(Q) + 2 \eps \kappa Q^2\bigr) \bigr], \\[1mm]
    	&\dot{p} = m_2(Q) + \eps \kappa Q^2,\\[1mm]
	&\dot{f}  = m_1(s)+ M_2(Q), 
\end{aligned}
\end{cases}
\text{where}\
\begin{cases}
m_1(s) = \frac{s}{K + s},\\[1mm]
M_2(Q) = \frac{Q}{1 + \eps Q}= Q -\eps Q^2 +O(\eps^2),\\[1mm]
\eps = \frac{1}{V}.
\end{cases}
\end{equation}
Note that with $q=O(\eps)$, DG hydrolysis is in zero order approximation a linear term (since there is little substrate) and that DG transacylation is a term of order $O(\eps)$. As a consequence, the asymptotic expansion in the case 
$V$ large is simpler and more explicit than in the previous case of $L$ large. 
We insert the ansatz $Q = Q_0 + \eps Q_1 + O(\eps^2)$ into the $Q$ equation in \eqref{eq:fullV} 
and compare the coefficients to obtain
\begin{equation}\label{eq:Qzeroone}
Q_0(t) = m_1(s(t)), \qquad Q_1 = - \frac{\dot{Q}_0}{L}  + \left[ 1 - 2 \kappa \right]  Q_0^2 
= -\frac{m_1'}{L} \dot{s} + \left[ 1 - 2 \kappa \right] m_1^2(s).
\end{equation}

Inserting \eqref{eq:Qzeroone} into \eqref{eq:fullV} yields the first order QSSA system
\begin{equation}\label{eq:approxV}
\begin{cases}
\begin{aligned}
     	&\dot{s} = -m_1(s) + \eps \kappa m_1^2(s),\\[1mm]
    	&\dot{p} = m_1(s) + \eps \frac{m_1'(s) m_1(s)}{L} - \eps \kappa m_1^2(s),\\[1mm]
	&\dot{f}  = 2 m_1(s)+ \eps  \frac{m_1'(s) m_1(s)}{L} - 2\eps \kappa m_1^2(s),
\end{aligned}
\end{cases}
\quad\text{where}\quad
\begin{cases}
m_1(s) = \frac{s}{K + s},\\[1mm]
\eps = \frac{1}{V}.
\end{cases}
\end{equation}

Note, that solutions to \eqref{eq:approxV} conserves total glycerol and total FAs only up to the zero order terms. In first order, the conservation laws do not hold due to the term $\eps m_1' m_1(s)$ in the equations for $\dot{p}$ and $\dot{f}$, which stems form the first term in $Q_1$ and is hence proportional to $\dot{s}$ and a results of the asymptotic expansion. 
The $s$ equation in \eqref{eq:approxV} is scalar like in system \eqref{eq:firstorderL} and has the added benefit of not requiring to solve for $\tilde{q}$.  

\begin{figure}[htb]
\centering
	\includegraphics[width=1\textwidth]{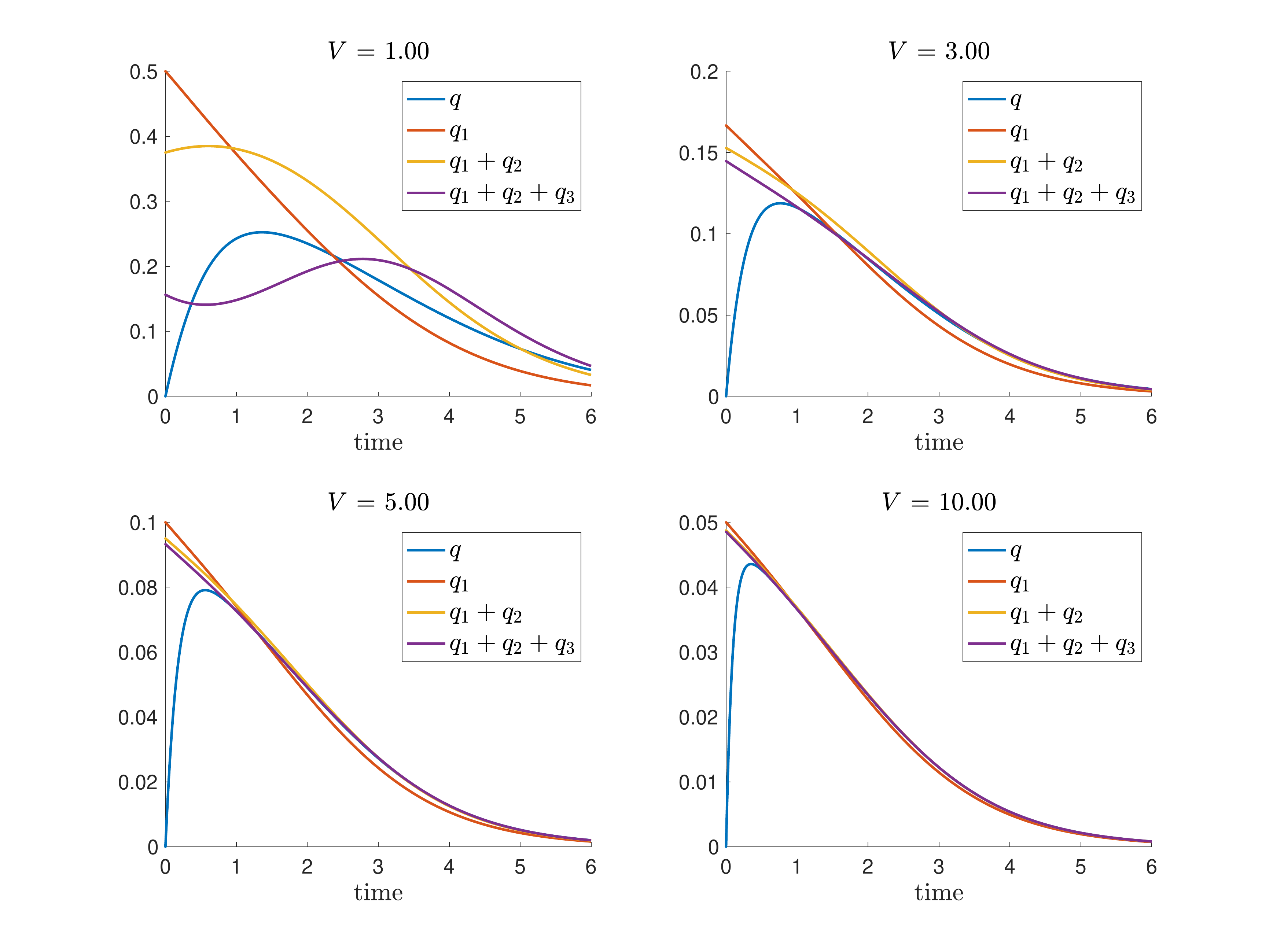}
	\caption[]{Asymptotic expansion of $q$ given by \eqref{eq:10} for different values of $V$ and fixed $L = K = \kappa = 1$.}
	\label{fig:09b}
\end{figure}

We again illustrate the asymptotic analysis by numerical simulations, which plot the asymptotic expansion of 
the QSSA quantity $q= \eps Q$ and $\eps = 1/V$. Unfortunately, as of now we are lacking experimental data, which would allow to identify $V$ from experiments. However, an educated guess suggest that values $V\sim3$ could be expected. 
Hence, we provide the asymptotic expansion of $q$ up to third order terms. Straightforward (and somewhat lengthy) 
calculations show that the following asymptotic expansion of $q$:
\begin{align} 
q &= \underbrace{\frac{m_1(s)}{V}}_{=: q_1 }  + \underbrace{\frac{m_1(s)}{V^2}  \left( - (2\kappa -1) m_1(s) + \frac{1}{L} m_1'(s)\right)}_{=: q_2 } \hspace{5cm} & \nonumber \\  
&\quad + \underbrace{ \frac{m_1(s)}{V^3} \left( \frac{4-9\kappa }{L} m_1'(s) m_1(s) + (2(2\kappa-1)^2-1)m_1^2(s) + \frac{m_1''(s) m_1(s) + (m_1'(s))^2}{L^2}\right)}_{=: q_3}\nonumber \\
&\quad  + O(1/V^4) \label{eq:10} 
\end{align}
\begin{remark}
Note that in the asymptotics of large $V$, the zero order approximation $q_0=0$, which follows from expanding 
the rescaled system \eqref{eq:fullV}, but could have already been see from the $q$ equation in \eqref{eq:full}: 
There, for large $V$ follows already that $q=O\bigl(\tfrac{1}{V}\bigr)$ holds.
\end{remark}

Figures \ref{fig:09b} plots $q(t)$ as part of the solution of \eqref{eq:full} and shows in comparison the approximations by $q_1$, $q_1 + q_2$ and $q_1 + q_2 + q_3$ given by \eqref{eq:10} for parameters $L = K = \kappa = 1$ and where $s(t)$ is part of the solution of \eqref{eq:full}. We see that already for $V=3$, the solution $q$ and the third order approximation $q_1 + q_2 + q_3$ 
are in very good agreement after the initial layer, while this is not really true for the first and second order approximation. 
Even better approximation is obtained fir larger values of $L$.


\subsection{First order asymptotic expansion of $q$: $\kappa$ large}
We consider the case $\kappa$ large. This is a parametric regime where $\sigma K_2^2\gg V_2$ holds and DG transacylation, which is a loss term in the $q$ equation in \eqref{eq:full} is a fast process. 
Like in the case of $V$ large, $q$ is again expected to be small. Given the quadratic nonlinearity of transacylation, it turns out that the correct rescaling is $q = \eps Q$ with $\eps = 1/\sqrt{\kappa}$. By accordingly rescaling the non-dimensionalised system \eqref{eq:full}, we obtain 
\begin{equation}\label{eq:fullkappa}
\begin{cases}
\begin{aligned}
     	&\dot{s} = -m_1(s) +  V Q^2,\\[1mm]
    	&\eps \dot{Q} = L \left[ m_1(s) - V \left(M_2(Q) + 2 Q^2\right) \right], \\[1mm]
    	&\dot{p} = V \left(M_2(Q) +  Q^2\right),\\[1mm]
	&\dot{f}  = m_1(s)+ V M_2(Q), 
\end{aligned}
\end{cases}
\quad\text{where}\quad
\begin{cases}
m_1(s) = \frac{s}{K + s},\\[1mm]
M_2(Q) = \frac{\eps Q}{1 + \eps Q}= \eps Q + O(\eps^2)\\[1mm]
\eps = \frac{1}{\sqrt{\kappa}}.
\end{cases}
\end{equation}
Note that the this rescaling DG transacylation is dominant and or order $O(1)$ while DG hydrolysis is a $O(\eps)$ term and linear up to order $O(\eps^2)$ (since the substrate concentration $q$ is small). Hence, the asymptotic expansion in this case is again different to the cases $V$ or $L$ large. 

We apply the asymptotic expansion $Q = Q_0 + \eps Q_1 + O(\eps^2)$. Insertion into \eqref{eq:fullkappa}
and comparison of the coefficients yields 
\begin{equation}\label{eq:Qzeroonekappa}
Q_0(t) =  \sqrt{\frac{m_1(s(t))}{2V}}, \qquad \text{and}\qquad Q_1 = - \frac{1}{4}-\frac{1}{8 L V} \frac{m_1'(s)}{m_1(s)} \dot{s}  = - \frac{1}{4}
-\frac{m_1'(s)}{16 L V} +O(\eps),  
\end{equation}
where we expressed $\dot{s}$  by means of the $s$ equation in \eqref{eq:fullkappa}.
Inserting \eqref{eq:Qzeroonekappa} into \eqref{eq:fullV} yields the first order approximative system
\begin{equation}\label{eq:approxkappa}
\begin{cases}
\begin{aligned}
     	&\dot{s} = -\frac{m_1(s)}{2}  - \eps \sqrt{\frac{V m_1(s)}{8}} + \eps\frac{m_1^{\prime}(s)}{8L}\sqrt{\frac{m_1(s)}{2  V}} + O(\eps^2),\\[1mm]
    	&\dot{p} = \frac{m_1(s)}{2}  + \eps\sqrt{\frac{V m_1(s)}{8}} + \eps\frac{m_1^{\prime}(s)}{8L}\sqrt{\frac{m_1(s)}{2 V}} +   O(\eps^2).\\[1mm]
	&\dot{f}  = m_1(s)+ \eps \sqrt{\frac{V m_1(s)}{2}}+ O(\eps^2),
\end{aligned}
\end{cases}
\quad\text{where}\quad
\begin{cases}
m_1(s) = \frac{s}{K + s},\\[1mm]
\eps = \frac{1}{\sqrt{\kappa}}.
\end{cases}
\end{equation}

 \begin{figure}[tb]
\centering
	\includegraphics[width=1\textwidth]{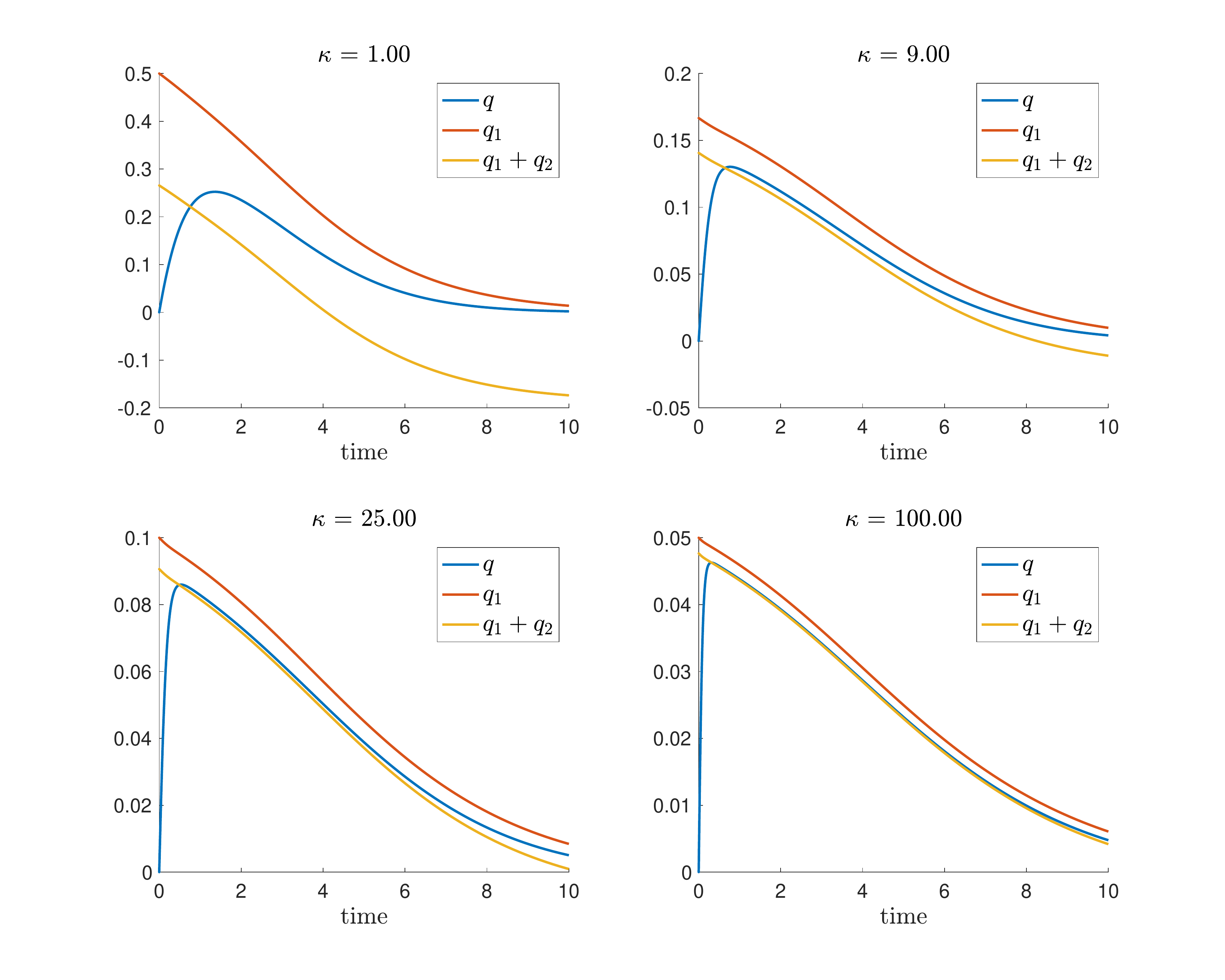}
	\caption[]{Asymptotic expansion of $q$ given by \eqref{eq:qkappa} for different values of $\kappa$ and fixed $L= K = V = 1$.}
	\label{fig:09f}
\end{figure}

We illustrate the asymptotic analysis by numerical simulations, which plot the asymptotic expansion of 
the QSSA quantity $q= \eps Q$ and $\eps = 1/\sqrt{\kappa}$. Rescaling \eqref{eq:Qzeroonekappa}
back yields  the following asymptotic expansion of $q$  up to second order terms, where (like in the case $V$ large)
the zero order approximation $q_0=0$:  
\begin{equation}\label{eq:qkappa} 
q = \underbrace{\frac{1}{\sqrt{\kappa}}\sqrt{\frac{m_1(s)}{2V}}}_{ =: q_1} + \underbrace{\frac{1}{\kappa}\left(\frac{m_1'(s)}{16 L V} - \frac{1}{4}\right)}_{=: q_2} + O\biggl(\frac{1}{\kappa^{3/2}}\biggr). 
\end{equation}
Figures \ref{fig:09f} plots $q(t)$ as part of the solution of \eqref{eq:full} and shows in comparison the approximations by $q_1$ and $q_1 + q_2$ given by \eqref{eq:qkappa} for parameters $L = K = V = 1$ and where $s(t)$ is part of the solution of \eqref{eq:full}. We see very good agreement only vor very large values like $\kappa=100$, which reflects that the scaling 
factor is $\sqrt{\kappa}$. Already with $\kappa~25$, the second order QSSA $q_1 + q_2$ is not longer a very good approximation.


\section{Conclusion and outlook}\label{sec:conclusion}
Intracellular lipolysis is a central catabolic pathway involved in a variety of cellular processes including energy production, signal transduction, and lipid remodelling. The main lipolytic enzymes (i.e. ATGL, HSL, and MGL) have been functionally characterised in great detail. Over the last decades, gain- and loss-of-function studies in different model organisms corroborated their crucial roles in neutral lipid catabolism \cite{Grabner-2021, Schreiber-2019}. Nevertheless, lipolysis is far from being completely understood. We are convinced that the classical view of lipolysis as a linear process involving the consecutive action of ATGL, HSL, and MGL is too simplistic. On the one hand, published data suggest that lipolysis is an enzymatically redundant process (ATGL not only hydrolyses TGs but also DGs \cite{Zimmermann-2004}, \textit{vice versa} HSL not only hydrolyses DGs but also TGs and MGs \cite{Fredrikson-1981}). On the other hand, three independent studies convincingly showed that ATGL also exhibits DG transacylation activity at least \textit{in vitro} \cite{Jenkins-2004, Zhang-2019, Kulminskaya-2021}, indicating that intracellular lipolysis is a non-linear process. The \textit{in vivo} relevance of DG transacylation is still elusive though. 

Kinetic parameters of model system \eqref{eq:02a} are currently not available.  In particular, the parameters $\kappa$ and $V$ for DG transacylation and DG hydrolysis are unknown since these two functions can not be distinguished by any current experimental protocol. In a work in progress, 
we are developing mathematical approaches to identify kinetic parameters from experimental data using minimisation of tracking functional as well as Bayesian estimators. 
First preliminary results of parameter identification (yet for more specific mathematical models than system \eqref{eq:02a}) suggest values of $\kappa$ and $V$, which put the system dynamics somewhere near the middle of the parametric plots Figs.~\ref{fig:50proa} and 
\ref{fig:50prob}. A plausible guess suggests $V$ somewhat larger that one and DG transacylation on average maybe 20\% of DG hydrolysis, but with a strong dependence on the substrate and an increase up to 50\% in some experiments. 

For these parameter values, the effects of DG transacylation as shown in Figs.~\ref{fig:50prob} are rather neutral: A slight delay in substrate processing time but negligible effects on the time of MG and FA production compared to the system without DG transacylation. 
This might be a reason why some researchers believe that DG transacylation does not play a significant role \textit{in vivo}: 
Not because it is not active, but because it behaves rather neutrally. 

However, if the hydrolytic machinery is shifted away from these 
conditions, DG trans\-acylation will play a more significant role. First, if 
DG hydrolysis is lacking, DG transacylation serves as a failsafe mechanism by producing MGs and even indirectly FAs
when TGs formed are hydrolysed further to DGs.
However, this failsafe mechanism comes at a price: 
If DG hydrolysis is present, a highly effective DG transacylation   
will compete with DG hydrolysis. This will cause downstream delays by forcing the lipolysis cascade TG $\to$ DG $\to$ MG rather into a semi-open loop of TG $\to$ DG $\to$ 1/2 TG $\to$ 1/2 DG $\to$ 1/4 TG etc, see e.g. 
Fig~\ref{fig:50prob} and it's illustration of the downstream delays.

It seems highly remarkable that our current preliminary guess on parameter values places the kinetics of the lipolytic machinery in the middle of these two options: failsafe mechanism versus DG substrate competition. One can speculate that the lipolytic machinery  evolved to operate at a flexible state rather than at the maybe more efficient state without DG transacylation. We interpret our observations that there is biological purpose behind both DG hydrolysis and DG transacylation constituting the lipolysis machinery. 
\medskip

As a second biological purpose, we can speculate that various effects of DG transacylation have the ability to homogenise the process of 
lipid hydrolysis over the surface of lipid droplets.
We recall that the ratio of DG transacylation to DG hydrolysis depends on the DG concentration levels and that DG transacylation becomes dominate for sufficiently large concentrations of DGs.  
The enzyme localisation on the surface of LDs can be temporally and spatially heterogeneous \cite{Egan-1992, Granneman-2009, Sztalryd-2003}. Local absence of HSL would lead to 
local DG accumulation as DGs can't leave the LD. It can by speculated that DG transacylation may serve as  a regulator of spatial heterogeneities of 
DG concentrations on the surface of the lipid droplet by processing local surpluses (and without having to wait for diffusion to spread DGs).
Hence, DG transacylation would contribute to an overall more homogeneous and controlled downstream processing, 
which is certainly an important physiological factor. However, ATGL-mediated DG transacylation averaged over the entire surface of the lipid droplet cannot fully compensate for complete HSL deficiency \textit{in vivo}, since HSL knockout mice accumulate DGs in adipose and non-adipose tissues \cite{Haemmerle-2002}. 

\medskip

There are many potential directions of future research. From a mathematical perspective, a rigorous proof 
of the convergence of the fast-slow ODE systems to the limiting QSSA is work in progress. 
The proof requires a sufficiently sharp perturbation argument 
for nonlinear, non-autonomous ODE systems. Unfortunately, estimations of the fundamental matrix of non-autonomous systems like in \cite[Chapter 3]{teschlode} are currently insufficient, which we strongly consider to be a technical artefact rather than the results of the 
system behaviour. As a by-pass, we are currently trying to find a suitable change of variables, which might allow to improve the estimates on
the non-autonomous semigroup. 
\medskip

From a modelling and interdisciplinary perspective, a central goal is the identification of the lipolytic kinetic parameters from experimental data. 
In particular, we aim to understand the dynamics of TG breakdown in adipocytes via cell-based assays and mice experiments. 
However, we will have to work with a limited number of data points both in adipocyte cell assays and even less in mice experiments. We expect that reduced QSSA models like presented in Section~\ref{sec:asymptotics} will be important in order to 
fit parameter from \textit{in vivo} data. 

There is, however, an additional mathematical question concerning the use of QSSAs. As part of the process of 
mathematical modelling, a sensitivity analysis of the parameter  is an important step. 
In particular, we are studying parameter sensitivity derivatives and how they evolve along solutions of the system. 
Parameter sensitivity derivatives provided very useful insights into the qualitative behaviour of the system and also into how challenging  the inverse problem of identifying parameters from measurements will be. 

In the spirit of this paper, which is to discuss the role of DG transacylation, the following system considers  the sensitivity derivatives with respect to the parameter $\kappa$:

\begin{equation}\label{eq:03}
\begin{cases}
\begin{aligned}
     	&\left(\frac{ds}{d\kappa}\right)^{\displaystyle \cdot} = -m_1'(s) \frac{ds}{d\kappa}  + V q^2 + 2V\kappa q \frac{dq}{d\kappa},\\
    	&\left(\frac{dq}{d\kappa}\right)^{\displaystyle \cdot} = L \left( m_1'(s) \frac{ds}{d\kappa} - 2V q^2 - V \left( m_2'(q) + 4 \kappa q \right) \frac{dq}{d\kappa} \right), \\
    	&\left(\frac{dp}{d\kappa}\right)^{\displaystyle \cdot} = V \left( q^2 + \left( m_2'(q) + 2 \kappa q \right) \frac{dq}{d\kappa} \right), \\
	&\left(\frac{df}{d\kappa}\right)^{\displaystyle \cdot}  = \left( m_1'(s) \frac{ds}{d\kappa} + V m_2'(q) \frac{dq}{d\kappa} \right), \\
\end{aligned}
\end{cases}
\end{equation}
subject to homogeneous initial data
\[ \frac{ds}{d\kappa} \Big\vert_{t = 0} = \frac{dq}{d\kappa} \Big\vert_{t = 0} = \frac{dp}{d\kappa} \Big\vert_{t = 0} = \frac{df}{d\kappa} \Big\vert_{t = 0} = 0. 
\]

Note that, for instance, if $\frac{ds}{d\kappa}$ is positive, then the dynamics observed for increased values of $\kappa$ will show also increased values of $s$. 

In Section~\ref{sec:asymptotics}, we have derived simplified QSSAs models. An interesting questions compares the parameter sensitivity derivatives of the full system 
\eqref{eq:03} with the the parameter sensitivity derivatives obtained from the QSSA. Note that, the QSSA $\dot{q} = 0$ implies $\left( dq / d\kappa \right)^{\displaystyle \cdot} = 0$ and, thus, $\left( dq / d\kappa \right) = 0$ for all times. Hence, from the second equation in \eqref{eq:03} we obtain an algebraic equation for the sensitivity of $\tilde{q}$ with respect to $\kappa$
\begin{equation}\label{eq:06} 
\frac{d \tilde{q}}{d\kappa}  = \frac{1}{\mu(\tilde{q})} \left( \frac{m_1'(\tilde{s})}{V} \frac{d\tilde{s}}{d\kappa} - 2 \tilde{q}^2\right),
\qquad\text{where}\quad \mu(\tilde{q}) := m_2'(\tilde{q}) + 4 \kappa \tilde{q}>0.
\end{equation}
We remark that \eqref{eq:06} follows equally from differentiating \eqref{QSSA} with respect to $\kappa$. By substituting \eqref{eq:06} into the equation for $ds/d\kappa$, we get in the QSSA system that
\[\begin{aligned}
\left(\frac{d \tilde{s}}{d\kappa}\right)^{\displaystyle \cdot} & = - m_1'(\tilde{s})\left( 1 - \frac{2\kappa\tilde{q}}{\mu(\tilde{q})}\right)\frac{d \tilde{s}}{d\kappa} + V\tilde{q}^2 \left( 1 -  \frac{4\kappa\tilde{q}}{\mu(\tilde{q})}\right) \\
& = -\underbrace{m_1'(\tilde{s})}_{\in [1/(K+1)^2,1/K^2]}\underbrace{\frac{2\kappa\tilde{q} + m_2'(\tilde{q})}{\mu(\tilde{q})}}_{\in [1/2,1]} \frac{d \tilde{s}}{d\kappa} + V\underbrace{\tilde{q}^2 \frac{m_2'(\tilde{q})}{\mu(\tilde{q})}}_{\in [0, 1/(8\kappa)]}
\end{aligned}\]
Since $d \tilde{s}/d\kappa = 0$ at $t = 0$ and $\tilde{q} > 0$, we deduce that 
\[ \left( \frac{d \tilde{s}}{d\kappa}\right)^{\displaystyle \cdot}  \approx  \frac{V \tilde{q}^2}{1 + 4\kappa\tilde{q} (1+\tilde{q})^2} > 0
\]
for small times.
Therefore $d \tilde{s}/d\kappa$ grows initially until the negative linear term dominates, which altogether implies that $d \tilde{s}/d\kappa$ remains positive for all times. 
We deduce from differentiating the conservation law \eqref{eq:02c} that
\[ 
\frac{d \tilde{p}}{d\kappa} = - \frac{d \tilde{s}}{d\kappa}
\]
and $d \tilde{p}/d\kappa$ is negative for all times, initially decreases until eventually approaches 0, see Fig.~\ref{fig:04e}

\begin{figure}[tb]
\centering
	\includegraphics[width=0.9\textwidth]{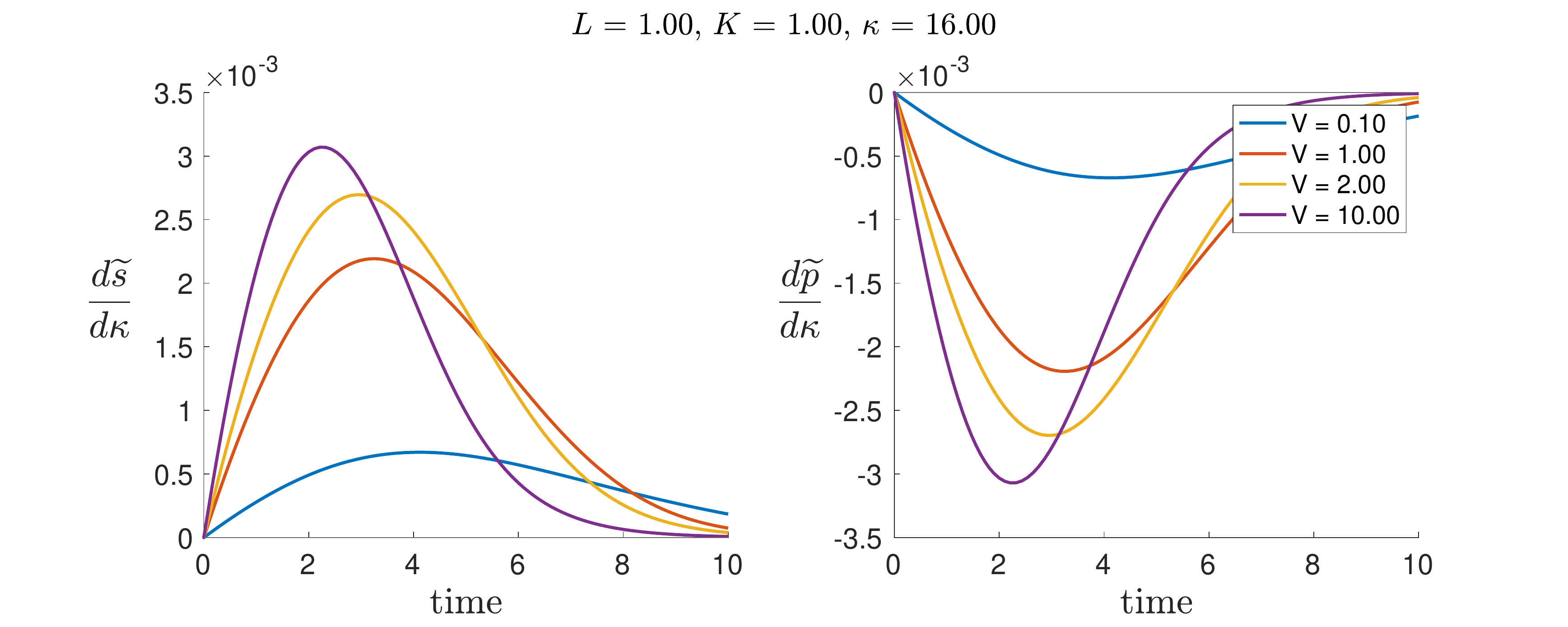}
	\caption[]{Other parameters: $L = K = 1$, $\kappa = 16$.}
	\label{fig:04e}
\end{figure}

Figures~\ref{fig:04e} plots the parameter sensitivity derivatives 
$d \tilde{s}/d\kappa$ and $d \tilde{p}/d\kappa$
of the QSSA model for the parameters $L=K=1$ and $\kappa=16$, 
which illustrates this analysis for different values of $V=0.1,1,2,10$.

\begin{figure}[!htb]
\centering
	\includegraphics[width=1\textwidth]{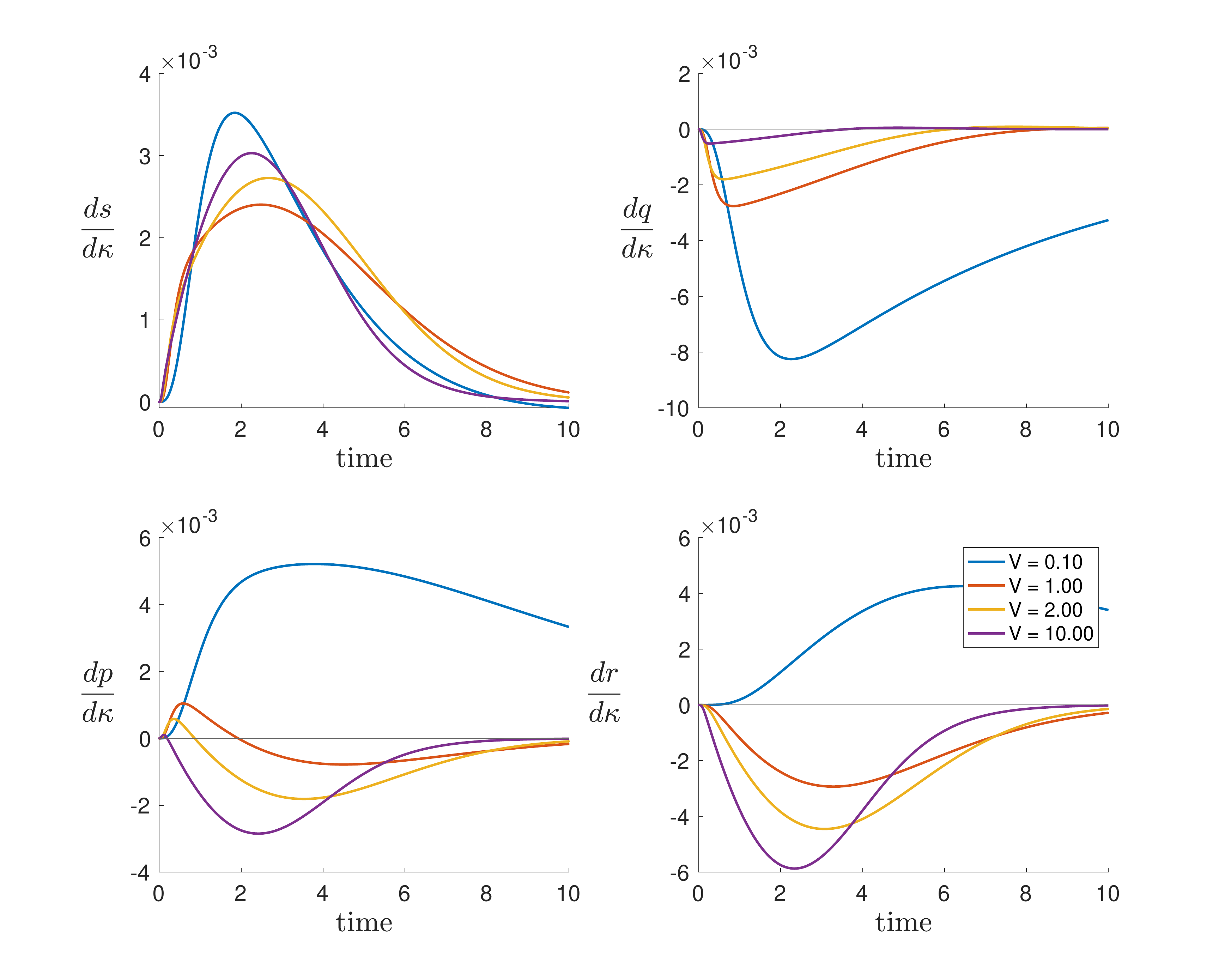}
	\caption[]{Other parameters: $L = K = 1$, $\kappa = 16$.}
	\label{fig:04a}
\end{figure}

In comparison, Figures~\ref{fig:04a} plots the evolution of the $\kappa$-sensitivity derivatives for the full system \eqref{eq:03} for parameters $L = K = 1$, $\kappa = 16$ and the same values of $V=0.1,1,2,10$. We observe that the qualitative behaviour of $\tfrac{ds}{d\kappa}$ and 
$\tfrac{d \tilde{s}}{d\kappa}$ corresponds well. 
However, in the cases $V=0.1,1,2$, the product sensitivity derivative 
$\frac{d \tilde p}{d\kappa}$ of the full system \eqref{eq:03} is positive for some 
initial time interval. This is in contrast to our analysis of the QSSA system, 
where $\frac{d \tilde p}{d\kappa}\le0$ holds. Hence, there are parameters for which the full system predicts an increase of MG production due to increased DG transacylation when the QSSA system predicts a decrease of MG production due to 
increased DG transacylation. In particular, this contradiction is observed for parameter value $V\le2$, where for $V=2$ Fig.~\ref{fig:QSSA} plots (for the same parameters $L=K=1$ and $\kappa=16$) a seemingly very good approximation of $q$ by $\tilde{q}$. 
This example of disagreement between full and QSSA model suggests that an analysis of validity of a QSSA should 
also include the sensitivity derivatives.
\bigskip

%

\noindent{\bf Acknowledgements.}\hfill

This work is supported by 
SFB Lipid hydrolysis (F 73) funded by the Austrian Science Fund FWF, 
by the International Research Training Group IGDK 1754 and NAWI Graz.
The authors greatfully acknowledge discussions with Natalia Kulminskaya.
\bibliography{references}

\end{document}